\documentclass[english]{amsart}

\usepackage{amsmath}

\usepackage{amssymb}

\usepackage{amscd} \usepackage{tabularx}
\usepackage[all,cmtip,line]{xy}

\newcommand{\ra}{\rightarrow}

\newcommand{\into}{\hookrightarrow}

\newcommand{\gothgl}{{\mathfrak{gl}}}

\newlength{\ownl}

\newcommand{\ndiv}{{\mbox{$\not| $}}}

\newcommand{\Art}{{\operatorname{Art}\,}}

\newcommand{\Fil}{{\operatorname{Fil}\,}}

\newcommand{\Frob}{{\operatorname{Frob}}}

\newcommand{\Hom}{{\operatorname{Hom}\,}}

\newcommand{\Ind}{{\operatorname{Ind}\,}}

\newcommand{\Res}{{\operatorname{Res}}}

\newcommand{\WD}{{\operatorname{WD}}}

\newcommand{\Spec}{{\operatorname{Spec}\,}}

\newcommand{\gr}{{\operatorname{gr}\,}}

\newcommand{\rec}{{\operatorname{rec}}}

\newcommand{\diag}{{\operatorname{diag}}}

\newcommand{\Wdiag}{W^{\diag}}
\newcommand{\Wcris}{W^{\cris}}

\newcommand{\Wexplicit}{W^{\operatorname{explicit}}}

\newcommand{\cris}{{\operatorname{cris}}}

\newcommand{\ab}{{\operatorname{ab}}}

\newcommand{\nr}{{\operatorname{nr}}}

\newcommand{\semis}{{\operatorname{ss}}}

\newcommand{\univ}{{\operatorname{univ}}}

\newcommand{\A}{{\mathbb{A}}}

\newcommand{\C}{{\mathbb{C}}}

\newcommand{\F}{{\mathbb{F}}}

\newcommand{\Q}{{\mathbb{Q}}}
\newcommand{\R}{{\mathbb{R}}}

\newcommand{\Z}{{\mathbb{Z}}}

\newcommand{\CO}{{\mathcal{O}}}

\newcommand{\gm}{{\mathfrak{m}}}

\newcommand{\barK}{\overline{{K}}}

\newcommand{\barFF}{\overline{{\F}}}

\newcommand{\barQQ}{\overline{{\Q}}}

\newcommand{\tI}{\widetilde{{I}}}

\newcommand{\tS}{\widetilde{{S}}}

\newcommand{\tv}{{\widetilde{{v}}}}

\newcommand{\epsilonbar    }{\overline{\epsilon}}

 \newcommand{\barrho   }{{\overline{\rho}}}

\newcommand{\Qbar}{{\overline{\Q}}}

\def\RCS$#1: #2 ${\expandafter\def\csname RCS#1\endcsname{#2}}
\RCS$Revision: 43 $
\RCS$Date: 2014-05-12 09:03:30 -0400 (Mon, 12 May 2014) $

\newcommand{\To}{\longrightarrow}
\newcommand{\m}{\mathfrak{m}}

\newcommand{\isoto}{\stackrel{\sim}{\To}} 

\newcommand{\bigO}{\mathcal{O}} 

\newcommand{\bb}{\mathbb} 
\newcommand{\mc}{\mathcal}
\newcommand{\mf}{\mathfrak}

\newcommand{\cA}{\mathcal{A}}

\newcommand{\cG}{\mathcal{G}}

\newcommand{\cO}{\mathcal{O}}

\newcommand{\rhobar}{\overline{\rho}} 
 
\newcommand{\rbar}{\bar{r}}

\newcommand{\mubar}{\overline{\mu}}

\newcommand{\GL}{\operatorname{GL}}

\newcommand{\HT}{\operatorname{HT}}

 \newcommand{\Qp}{\Q_p}
\newcommand{\Ql}{{\Q_l}} 
\newcommand{\GQl}{{G_\Ql}}

\newcommand{\Qpbar}{\overline{\Q}_p}
\newcommand{\Qlbar}{{\overline{\Q}_{l}}}

\newcommand{\Fl}{{\F_l}}
\newcommand{\Flbar}{{\overline{\F}_l}}

\def\numequation{\addtocounter{subsubsection}{1}\begin{equation}}

\def\Fss{\mathrm{F-ss}}
\def\univ{\mathrm{univ}}
\usepackage{amsthm}

 \newtheorem{ithm}{Theorem}
\newtheorem{thm}{Theorem}[subsection]

\newtheorem{cor}[thm]{Corollary}
 \newtheorem{lemma}[thm]{Lemma}
\newtheorem{lem}[thm]{Lemma} 
 \theoremstyle{definition}
 \theoremstyle{definition}
\newtheorem{defn}[thm]{Definition} \theoremstyle{remark}
\newtheorem{rem}[thm]{Remark} 
\newtheorem{remark}[thm]{Remark} 
 
\numberwithin{equation}{subsection}

\theoremstyle{definition}

\begin{document}
\title{Serre weights for $U(n)$.}

\author{Thomas Barnet-Lamb}\email{tbl@brandeis.edu}\address{Department of Mathematics, Brandeis University}
\author{Toby Gee} \email{toby.gee@imperial.ac.uk} \address{Department of
  Mathematics, Imperial College London}\author{David Geraghty}
\email{david.geraghty@bc.edu}\address{Department of Mathematics, 
  Boston College}  \thanks{The second author was partially supported
  by NSF grant DMS-0841491, a Marie Curie Career Integration Grant, and by an
  ERC Starting Grant, and the third author was partially supported
  by NSF grant DMS-1440703.}  

\begin{abstract}We study the weight part of (a generalisation of)
  Serre's conjecture for mod $l$ Galois representations associated to
  automorphic representations on unitary groups of rank $n$ for odd
  primes $l$. Given a modular Galois representation, we use automorphy
  lifting theorems to prove that it is modular in many other
  weights. We make no assumptions on the ramification or inertial
  degrees of $l$. We give an explicit strengthened result when $n=3$
  and $l$ splits completely in the underlying CM field.
\end{abstract}
\maketitle
\tableofcontents

\section{Introduction}\label{sec:intro}

In recent years there has been considerable progress in formulating
generalisations of Serre's conjecture, and in particular of the weight
part of Serre's conjecture, for higher-dimensional groups;
cf. \cite{MR1896473}, \cite{herzigthesis}, \cite{gee061},
\cite{GHS}. There has been rather less progress in proving cases of
these conjectures; indeed, the only results that we are aware of are
the essentially complete treatment of the ordinary case for definite
unitary groups in \cite{gg}, and the results of \cite{egh} for definite
unitary groups of rank $3$. 

In the present paper, we use the automorphy lifting theorems developed
in \cite{blgg}, \cite{blggord} and \cite{BLGGT} to prove that a
modular Galois representation, coming from an automorphic form on $U(n)$, is necessarily modular in a number of
additional weights predicted by the conjectures of \cite{herzigthesis}
and \cite{GHS}. Rather complete results are available in the case
$n=2$, which are worked out in detail in the papers
\cite{BLGGU2,GLS12,GLS13}, so we concentrate in this paper on the case that
$n>2$. The additional complications are twofold. Firstly, we no longer
know that any modular Galois representation admits a potentially
diagonalizable lift (in the case $n=2$, this is proved in
\cite{BLGGU2} as a consequence of the results of \cite{kis04} and
\cite{MR2280776}). Secondly, the relationship between being modular of
some weight and having an automorphic lift of some weight is
substantially more complicated for $n>2$ than it is for $n=2$; in
particular, it is no longer the case that given an irreducible mod $l$
representation  $F$ of $\GL_n(\Fl)$, there is necessarily an
irreducible characteristic zero algebraic representation $W$ of
$\GL_n$ whose reduction modulo $l$ is $F$. Instead, one finds
that $F$ is the socle of the reduction modulo $l$ of some $W$, and this
gives strictly weaker information.

As a result of these two difficulties, our main theorems have two
restrictions. Let $F$ be a CM field, and let
$\rbar:G_F\to\GL_n(\Flbar)$ be our given modular Galois
representation. Firstly, we must assume that $\rbar$ has a potentially
diagonalizable automorphic lift. This assumption is perhaps not as
serious as it initially sounds, as it is conjecturally always satisfied, and in
particular is known to hold provided that $l$ is
unramified in $F$ and $\rbar$ has an automorphic lift of sufficiently
small weight. Secondly, rather than prove that $\rbar$ is modular of
some particular weight, we typically only provide a list of weights,
and guarantee that $\rbar$ is modular of some weight in this
list. In fact, it is often the case that only one weight on this list
is predicted by the conjectures of \cite{herzigthesis} and \cite{GHS},
and it should presumably be possible to prove modularity in this
weight in many cases using integral $p$-adic Hodge theory. We carry
out such an analysis in detail in the case $n=3$, defining a list of
conjectural weights $\Wexplicit(\rbar)$, and obtaining the following
result (Theorem \ref{thm: explicit result for GL_3 in l split
  completely case}).
\begin{ithm}\label{thm: A}Let $F$ be an imaginary CM field with maximal totally real subfield
  $F^+$, and suppose that $F/F^+$ is unramified at all finite places,
  and that $l$ splits completely in $F$. Suppose that $l>2$, and that
  $\rbar:G_F\to\GL_3(\Flbar)$ is an irreducible representation with
  split ramification. Assume that there is a RACSDC automorphic representation $\Pi$ of
  $\GL_3(\A_F)$ of weight $\mu\in(\Z^3_+)_0^{\Hom(F,\C)}$ and level
  prime to $l$ such that
  \begin{itemize}
  \item $\rbar\cong\rbar_{l,\imath}(\Pi)$ (so in particular, $\rbar^c\cong\rbar^\vee\epsilonbar_l^{-2}$).
  \item For each $\tau\in\Hom(F,\C)$, $\mu_{\tau,1}-\mu_{\tau,3}\le
    l-3$.
     \item $\rbar(G_{F(\zeta_l)})$ is adequate.
  \end{itemize}
 Let $a\in(\Z^3_+)_0^{\coprod_{w|l}\Hom(k_w,\Flbar)}$ be a generic
  Serre weight. Assume that $a\in \Wexplicit(\rbar)$.
Then $\rbar$ is modular of weight $a$.
  
\end{ithm}
(See sections \ref{sec:Definitions} and \ref{sec:Conjectures} for any
unfamiliar terminology, and section~\ref{GL3 results with l
  split} for the definition of ``generic'' that we are using, which is extremely
mild.) We should point out that we do \emph{not}
expect that $\Wexplicit(\rbar)$ contains all the weights in which
$\rbar$ is modular; rather, it consists of those weights which are
``obvious'' in the terminology of \cite{GHS}. (It is perhaps worth
remarking that despite the name, it is not obvious that $\rbar$ is
modular in any of these weights!) In order to prove this theorem we
make use of Fontaine-Laffaille theory; it seems likely that if one
could compute the possible reductions of crystalline Galois
representations outside of the Fontaine-Laffaille range then one could
prove an analogous theorem for $n>3$.

We now outline the structure of this paper. In Section
\ref{sec:Definitions} we define the spaces of automorphic forms that
we work with, and define what it means for $\rbar$ to be modular of
some weight. In Section \ref{sec:A lifting
  theorem} we establish the main lifting theorem that we need, a
corollary of the results of \cite{BLGGT}. In Section
\ref{sec:Conjectures} we define the set of weights
$\Wexplicit(\rbar)$, recall some results from Fontaine-Laffaille
theory, and prove our main results for arbitrary $n$. Finally, in
Section \ref{GL3 results with l
  split} we prove Theorem \ref{thm: A}.

\subsection{Notation}If $M$ is a field, we let $G_M$ denote its absolute Galois group.  We
write all matrix transposes on the left; so ${}^tA$ is the transpose
of $A$. Let $\epsilon_l$ denote the $l$-adic cyclotomic character, and
$\bar{\epsilon}_l$ or $\omega_l$ the mod $l$ cyclotomic character. If $M$ is
a finite extension of $\bb{Q}_p$ for some $p$, we write $I_M$ for the
inertia subgroup of $G_M$. If $R$ is a local ring we write
$\mf{m}_{R}$ for the maximal ideal of $R$.

We fix an algebraic closure $\Qbar$ of $\Q$. For each prime $p$ we fix
an algebraic closure $\Qpbar$ of $\Qp$, and we fix an embedding
$\Qbar\into\Qpbar$.

If $W$ is a de Rham representation of $G_K$ over
$\barQQ_l$ and if $\tau:K \into \barQQ_l$ then by definition the multiset
$\HT_\tau(W)$ of Hodge-Tate weights of $W$ with respect to $\tau$ contains $i$ with multiplicity $\dim_{\barQQ_l} (W
\otimes_{\tau,K} \widehat{\barK}(i))^{G_K} $. Thus for example
$\HT_\tau(\epsilon_l)=\{ -1\}$.

If $K$ is a finite extension of $\Qp$ for some
$p$, we will let $\rec_K$ be the local Langlands correspondence of
\cite{ht}, so that if $\pi$ is an irreducible complex
admissible representation of $\GL_n(K)$, then $\rec_K(\pi)$ is a
Weil-Deligne representation of the Weil group $W_K$. We will write $\rec$ for $\rec_K$
when the choice of $K$ is clear. We write $\Art_K:K^\times\to W_K$ for
the isomorphism of local class field theory, normalised so that
uniformisers correspond to geometric Frobenius elements.

Let $K$ be a finite extension of $\Ql$ with residue field $k$. For
each $\sigma\in \Hom(k,\Flbar)$ we define the fundamental character
$\omega_{\sigma}$ corresponding to $\sigma$ to be the
composite $$\xymatrix{I_{K^{\ab}/K}\ar[r]^{\Art_K^{-1}} &
  \bigO_{K}^{\times}\ar[r] & k^{\times}\ar[r]^{\sigma^{-1}} &
  \Flbar^{\times}.}$$
For any algebraic extension $L$ of $\Ql$, we often denote by
$\Hom(K,L)$ the set of field homomorphisms from $K$ to $L$ which are
continuous for the $l$-adic topologies on $K$ and $L$ (or
equivalently, which are $\Ql$-linear).

\section{Definitions}\label{sec:Definitions} \subsection{}

 Let $l$ be a
prime, and let $F$ be an imaginary CM field with
maximal totally real field subfield $F^+$. We assume throughout this
paper that:
\begin{itemize}
\item $F/F^+$ is unramified at all finite places.
\item Every place $v|l$ of $F^+$ splits in $F$.
\item If $n$ is even, then $n[F^+:\Q]/2$ is also even.
\end{itemize}
Under these hypotheses, there is a reductive algebraic group $G/F^+$
with the following properties:
\begin{itemize}
\item $G$ is an outer form of $\GL_n$, with $G_{/F}\cong\GL_{n/F}$.
\item If $v$ is a finite place of $F^+$, $G$ is quasi-split at $v$.
\item If $v$ is an infinite place of $F^+$, then $G(F^+_v)\cong U_n(\R)$.
\end{itemize}
To see that such a group exists, one may argue as follows. Let $B$
denote the matrix algebra $M_n(F)$. An involution $\ddag$ of the
second kind on $B$ gives a reductive group $G_\ddag$ over $F^+$ by
setting \[G_\ddag(R)=\{g\in B\otimes_{F^+} R:g^\ddag g=1\}\] for any
$F^+$-algebra $R$. Any such $G_\ddag$is an outer form of $\GL_n$,
with $G_{\ddag/F}\cong\GL_{n/F}$. One can choose $\ddag$ such that
\begin{itemize}
 \item If $v$ is a finite place of $F^+$, $G_\ddag$ is quasi-split at $v$.
\item If $v$ is an infinite place of $F^+$, then $G_\ddag(F^+_v)\cong U_n(\R)$.
\end{itemize}
To see this, one uses the argument of Lemma I.7.1 of \cite{ht}. We then fix some choice of
$\ddag$ as above, and take $G=G_\ddag$.

As in section 3.3 of \cite{cht} we define a model for $G$ over
$\cO_{F^+}$ in the following way. We choose an order $\cO_B$ in $B$
such that $\cO_B^\ddag=\cO_B$, and $\cO_{B,w}$ is a maximal order in
$B_w$ for all places $w$ of $F$ which are split over $F^+$ (see
section 3.3 of \cite{cht} for a proof that such an order exists). Then
we can define $G$ over $\cO_{F^+}$ by setting \[G(R)=\{g\in
\cO_B\otimes_{\cO_{F^+}} R:g^\ddag g=1\}\] for any $\cO_{F^+}$-algebra $R$.

If $v$ is a place of $F^+$ which splits as $ww^c$ over $F$, then we
choose an isomorphism \[\iota_v:\cO_{B,v}\isoto
M_n(\cO_{F,v})=M_n(\cO_{F_w})\oplus M_n(\cO_{F_{w^c}})\] such that
$\iota_v(x^\ddag)={}^t\iota_v(x)^c$. This gives rise to an
isomorphism \[\iota_w:G(\cO_{F_v^+})\isoto \GL_n(\cO_{F_w})\] sending
$\iota_v^{-1}(x,{}^tx^{-c})$ to $x$.

Let $K$ be an algebraic extension of $\Ql$ in $\Qlbar$ which contains
the image of every embedding $F\into\Qlbar$, let $\cO$ denote the ring
of integers of $K$, and let $k$ denote the residue field of $K$. Let
$S_l$ denote the set of places of $F^+$ lying over $l$, and for each
$v\in S_l$ fix a place $\tv$ of $F$ lying over $v$. Let $\tilde{S}_l$
denote the set of places $\tilde{v}$ for $v\in S_l$.

Let $W$ be an $\cO$-module with an action of $G(\cO_{F^+,l})$, and let
$U\subset G(\A_{F^+}^\infty)$ be a compact open subgroup with the
property that for each $u\in U$, if $u_l$ denotes the projection of
$u$ to $G(F_l^+)$, then $u_l\in G(\cO_{F^+_l})$. Let $S(U,W)$ denote
the space of algebraic modular forms on $G$ of level $U$ and weight
$W$, i.e.\ the space of functions \[f:G(F^+)\backslash
G(\A_{F^+}^\infty)\to W\] with $f(gu)=u_l^{-1}f(g)$ for all $u\in U$.

Let $\tI_l$ denote the set of embeddings $F\into K$ giving rise to a
place in $\tS_l$. For any $\tv\in\tS_l$, let $\tI_\tv$ denote the set
of elements of $\tI_l$ lying over $\tv$. Let $\Z^n_+$ denote the set of tuples $(\lambda_1,\dots,\lambda_n)$ of
integers with $\lambda_1\ge \lambda_2\ge\dots\ge \lambda_n$. For any $\lambda\in\Z^n_+$, view
$\lambda$ as a dominant character of the algebraic group $\GL_{n/\cO}$ in
the usual way, and
let $M'_\lambda$ be the algebraic $\cO$-representation of $\GL_n$ given
by \[M'_\lambda:=\Ind_{B_n}^{\GL_n}(w_0\lambda)_{/\cO}\] where $B_n$ is the
standard Borel subgroup of $\GL_n$, and $w_0$ is the longest element
of the Weyl group (see \cite{MR2015057} for more details of these
notions). Write $M_\lambda$ for the $\cO$-representation of $\GL_n(\cO)$
obtained by evaluating $M'_\lambda$ on $\cO$. For any $\lambda\in(\Z^n_+)^{\tI_\tv}$, let $W_\lambda$
be the free $\cO$-module with an action of $\GL_n(\cO_{F_\tv})$
 given
by \[W_\lambda:=\otimes_{\tau\in\tI_\tv}M_{\lambda_\tau}\otimes_{\cO_{F_\tv},\tau}\cO.\]We
give this an action of $G(\cO_{F^+,v})$ via $\iota_\tv$. For any $\lambda\in(\Z^n_+)^{\tI_l}$, let $W_\lambda$
be the free $\cO$-module with an action of $G(\cO_{F^+,l})$ given
by \[W_\lambda:=\otimes_{\tv\in\tS_l} W_{\lambda_\tv}.\]
If $A$ is an $\cO$-module we let \[S_\lambda(U,A):=S(U,W_\lambda\otimes_\cO A).\]

For any compact open subgroup $U$ as above of $G(\A_{F^+}^\infty)$ we may write
$G(\A_{F^+}^\infty)=\coprod_i G(F^+)t_i U$ for some finite set
$\{t_i\}$. Then there is an isomorphism \[S(U,W)\to\oplus_i W^{U\cap
  t_i^{-1}G(F^+)t_i}\]given by $f\mapsto (f(t_i))_i$. 
We say that
$U$ is \emph{sufficiently small} if for some finite place $v$ of $F^+$
the projection of $U$ to $G(F^+_v)$ contains no element of finite
order other than the identity. Suppose that $U$ is sufficiently
small. Then for each $i$ as above we have
$U\cap t_i^{-1}G(F^+)t_i=\{1\}$, so taking $W=W_\lambda\otimes_\cO A$ we see
that for any $\cO$-module $A$, we
have \[S_\lambda(U,A)\cong S_\lambda(U,\cO)\otimes_\cO A.\]
We note when $U$ is not sufficiently small, we still have $S_\lambda(U,A)\cong
S_\lambda(U,\cO)\otimes_\cO A$ whenever $A$ is $\cO$-flat.

We now recall the relationship between our spaces of algebraic
automorphic forms and the space of automorphic forms on $G$. Write
$S_\lambda(\Qlbar)$ for the direct limit of the spaces $S_\lambda(U,\Qlbar)$ over compact
open subgroups $U$ as above (with the transition maps being the obvious
inclusions $S_\lambda(U,\Qlbar)\subset S_\lambda(V,\Qlbar)$ whenever $V\subset
U$). Concretely, $S_\lambda(\Qlbar)$ is the set of
functions \[f:G(F^+)\backslash G(\A_{F^+})\to W_\lambda\otimes_\cO\Qlbar\]
such that there is a compact open subgroup $U$ of
$G(\A_{F^+}^{\infty,l})\times G(\cO_{F^+,l})$
with \[f(gu)=u_l^{-1}f(g)\] for all $u\in U$, $g\in G(\A_{F^+})$. This
space has
a natural left action of $G(\A_{F^+}^\infty)$ via \[(g\cdot
f)(h):=g_lf(hg).\]

Fix an isomorphism $\imath:\Qlbar\isoto\C$. For each embedding
$\tau:F^+\into\R$, there is a unique embedding $\tilde{\tau}:F\into\C$
extending $\tau$ such that $\imath^{-1}\tilde{\tau}\in\tI_l$.
Let $\sigma_\lambda$ denote the representation of $G(F_\infty^+)$ given by
$W_\lambda\otimes_\cO\Qlbar\otimes_{\Qlbar,\imath}\C$, with an element $g\in
G(F^+_\infty)$ acting via $\otimes_\tau\tilde{\tau}(\iota_{\tilde{\tau}}(g))$. Let $\cA$ denote
the space of automorphic forms on $G(F^+)\backslash G(\A_{F^+})$. From
the proof of Proposition 3.3.2 of \cite{cht}, one easily obtains the following.

\begin{lem}
  \label{lem: relationship of algebraic automorphic forms to classical
    automorphic forms}There is an isomorphism of
  $G(\A_{F^+}^\infty)$-modules \[S_\lambda(\Qlbar)\isoto\Hom_{G(F^+_\infty)}(\sigma_\lambda^\vee,\cA).\]
\end{lem}In particular, we note that $S_\lambda(\Qlbar)$ is a semi-simple admissible $G(\A_{F^+}^\infty)$-module.

Following \cite{cht}, we say that a cuspidal automorphic
representation of $\GL_n(\A_F)$ is RACSDC (regular, algebraic, conjugate self dual, and cuspidal) if
\begin{itemize}
\item $\pi_\infty$ has the same infinitesimal character as some
  irreducible algebraic representation of $\Res_{F/\Q}\GL_n$, and
\item  $\pi^c\cong \pi^\vee$.
\end{itemize}
We say that $\pi$ has level prime to $l$ if $\pi_v$ is unramified for
all $v|l$. If $\Omega$ is an algebraically closed field of
characteristic 0 we write 
    $(\Z^n_+)_0^{\Hom(F,\Omega)}$ for the subset of elements
    $\lambda\in(\Z^n_+)^{\Hom(F,\Omega)}$ such
    that \[\lambda_{\tau,i}+\lambda_{\tau\circ c,n+1-i}=0\] for all $\tau$, $i$.

If $\lambda\in  (\Z^n_+)^{\Hom(F,\C)}$ we write $\Sigma_\lambda$ for the
irreducible algebraic representation of $\GL_n^{\Hom(F,\C)}$ given by
the tensor product over $\tau$ of the irreducible representations with
highest weights $\lambda_\tau$.
We say that a RACSDC automorphic representation $\pi$ of
$\GL_n(\A_F)$ has weight $\lambda\in(\Z^n_+)^{\Hom(F,\C)}$ if $\pi_\infty$ has the same
infinitesimal character as $\Sigma_\lambda^\vee$. If this is the case then
necessarily $\lambda\in(\Z^n_+)_0^{\Hom(F,\C)}$.

\begin{thm}
  \label{thm: existence of Galois reps attached to RACSDC}

If $\pi$ is a RACSDC automorphic representation of $\GL_n(\A_F)$ of
weight $\lambda$, then there is
a continuous semisimple
representation \[r_{l,\imath}(\pi):G_F\to\GL_n(\Qlbar)\] such that
\begin{enumerate}
\item $r_{l,\imath}(\pi)^c\cong
  r_{l,\imath}(\pi)^\vee\otimes\epsilon_l^{1-n}$.
\item The representation $r_{l,\iota}(\pi)$ is de Rham, and is
  crystalline if $\pi$ has level prime to $l$. If $\tau:F\into\Qlbar$
  then \[\HT_\tau(r_{l,\imath}(\pi))=\{\lambda_{\imath\tau,1}+n-1,\dots,\lambda_{\imath\tau,n}\}.\]
\item
 
For each finite place $v$ of $l$, we have \[\imath\WD(r_{l,\imath}(\pi)|_{G_{F_v}})^{\Fss}\cong
  \rec(\pi_v^\vee\otimes|\det|^{(1-n)/2}).\]Here
  $\WD(r_{l,\imath}(\pi)|_{G_{F_v}})^{\Fss}$ denotes the Frobenius
  semisimplification of the Weil-Deligne representation associated to
  $r_{l,\imath}(\pi)|_{G_{F_v}}$, as in section 1 of \cite{ty}.
\end{enumerate}

\end{thm}
\begin{proof}
  This follows at once from the main results of \cite{shin},
  \cite{chenevierharris}, \cite{ana}, \cite{blggtlocalglobalII} and \cite{ana-2}.
\end{proof}

We say that a continuous irreducible representation
$r:G_F\to\GL_n(\Qlbar)$ (respectively $\rbar:G_F\to\GL_n(\Flbar)$) is
automorphic if $r\cong r_{l,\imath}(\pi)$ (respectively $\rbar\cong\rbar_{l,\imath}(\pi)$) for some RACSDC
representation $\pi$ of $\GL_n(\A_F)$. We say that a continuous irreducible representation
$r:G_F\to\GL_n(\Qlbar)$ is automorphic of weight $\lambda\in
(\Z^n_+)_0^{\Hom(F,\Qlbar)}$ if $r\cong r_{l,\imath}(\pi)$ for some RACSDC
representation $\pi$ of $\GL_n(\A_F)$ of weight $\imath \lambda$.

The theory of base change gives a close relationship between
automorphic representations of $G(\A_{F^+})$ and automorphic representations of
$\GL_n(\A_F)$. For example, one has the
following consequences of Corollaire 5.3 and Th\'{e}or\`eme 5.4 of \cite{labesse}.
\begin{thm}\label{thm: base change from GL to G}
  Suppose that $\Pi$ is a RACSDC representation of $\GL_n(\A_F)$ of
  weight $\lambda\in(\Z_+^n)_0^{\Hom(F,\C)}$. Then there is an automorphic
  representation $\pi$ of $G(\A_{F^+})$ such that
  \begin{enumerate}
  \item For each embedding $\tau:F^+\into\R$ and each
    $\tilde{\tau}\into\C$ extending $\tau$, we have $\pi_\tau\cong
    \Sigma_{\lambda_{\tilde{\tau}}}^\vee\circ\iota_{\tilde{\tau}}$.
  \item If $v$ is a finite place of $F^+$ which splits as $ww^c$ in
    $F$, then $\pi_v\cong \Pi_w\circ\iota_w$.
  \item If $v$ is a finite place of $F^+$ which is inert in $F$, and
    $\Pi_v$ is unramified, then $\pi_v$ has a fixed vector for some
    hyperspecial maximal compact subgroup of $G(F^+_v)$.
  \end{enumerate}

\end{thm}

\begin{thm}
  \label{thm: base change from G to GL} Suppose that $\pi$ is an
  automorphic representation of $G(\A_{F^+})$. 
  Then either:
  \begin{enumerate}
  \item There is an RACSDC automorphic representation $\Pi$ of
    $\GL_n(\A_F)$ of some weight $\lambda\in(\Z^n_+)_0^{\Hom(F,\C)}$, or:
  \item There is a nontrivial partition $n=n_1+\dots+n_r$ and cuspidal
    automorphic representations $\Pi_i$ of $\GL_{n_i}(\A_F)$ such that if
    $\Pi:=\pi_1\boxplus\dots\boxplus\pi_r$ is the isobaric direct sum of 
the $\pi_i$, then $\Pi$ is regular, algebraic, and conjugate
    self-dual of some weight $\lambda\in(\Z^n_+)_0^{\Hom(F,\C)}$
\end{enumerate}such that in either case
    \begin{enumerate}
    \item For each embedding $\tau:F^+\into\R$ and each
    $\tilde{\tau}\into\C$ extending $\tau$, we have $\pi_\tau\cong
    \Sigma_{\lambda_{\tilde{\tau}}}^\vee\circ\iota_{\tilde{\tau}}$.
  \item If $v$ is a finite place of $F^+$ which splits as $ww^c$ in
    $F$, then $\pi_v\cong \Pi_w\circ\iota_w$.
  \item If $v$ is a finite place of $F^+$ which is inert in $F$, and
    $\pi_v$ has a fixed vector for some
    hyperspecial maximal compact subgroup of $G(F^+_v)$, then $\Pi_v$ is unramified.
    \end{enumerate}
\end{thm}

We now wish to define what it means for an irreducible representation
$\rbar:G_F\to\GL_n(\Flbar)$ to be modular of some weight. In order to
do so, we return to the spaces of algebraic modular forms considered
before. For each place $w|l$ of $F$, let $k_w$ denote the residue
field of $F_w$. If $w$ lies over a place $v$ of $F^+$, write
$v=ww^c$. Let $(\Z^n_+)_0^{\coprod_{w|l}\Hom(k_w,\Flbar)}$ denote the
subset of $(\Z^n_+)^{\coprod_{w|l}\Hom(k_w,\Flbar)}$ consisting of
elements $a$ such that for each $w|l$, if $\sigma\in\Hom(k_w,\Flbar)$
and $1\le i\le n$
then \[a_{\sigma,i}+a_{\sigma c,n+1-i}=0.\] We say that an
element  $a\in(\Z^n_+)_0^{\coprod_{w|l}\Hom(k_w,\Flbar)}$ is a \emph{Serre
  weight} if for each $w|l$ and each $\sigma\in\Hom(k_w,\Flbar)$ we
have \[l-1\ge a_{\sigma,i}-a_{\sigma,i+1}\] for all $1\le i\le n-1$.  Similarly, if $\F$ is
a finite extension of $\Fl$, we say that an element $a\in(\Z^n_+)^{\Hom(\F,\Flbar)}$ is a \emph{Serre
  weight} if for each $\sigma\in\Hom(\F,\Flbar)$ and each $1\le i\le n-1$ we
have \[l-1\ge a_{\sigma,i}-a_{\sigma,i+1}.\]

Given any $a\in\Z^n_+$ with $l-1\ge a_i-a_{i+1}$ for all $1\le i\le
n-1$, we define the $\F$-representation $P_a$ of $\GL_n(\F)$ to be the
representation obtained by evaluating $\Ind_{B_n}^{\GL_n}(w_0 a)_{\F}$
on $\F$, and let $N_a$ be the irreducible sub-$\F$-representation of $P_a$
generated by the highest weight vector (that this is indeed
irreducible follows for example from II.2.8(1) of \cite{MR2015057} and
the appendix to \cite{herzigthesis}).

If $a\in(\Z^n_+)^{\Hom(\F,\Flbar)}$ is a Serre weight then we define
an irreducible $\Flbar$-representation $F_a$ of $\GL_n(\F)$
by \[F_a:=\otimes_{\tau\in\Hom(\F,\Flbar)}N_{a_\tau}\otimes_{\F,\tau}\Flbar.\]
We say that two Serre weights $a$ and $b$ are \emph{equivalent} if and only if
$F_a\cong F_b$ as representations of $\GL_n(\F)$. This is equivalent
to demanding that for each $\sigma\in\Hom(\F,\Flbar)$, we
have \[a_{\sigma,i}-a_{\sigma,i+1}=b_{\sigma,i}-b_{\sigma,i+1},\] for each
$1\le i\le n-1$, and the
character $\F^\times\to\Flbar^\times$ given
by \[x\mapsto\prod_{\sigma\in\Hom(\F,\Flbar)}\sigma(x)^{a_{\sigma,n}-b_{\sigma,n}}\]is
trivial. Every irreducible $\Flbar$-representation of $\GL_n(\F)$ is
of the form $F_a$ for some $a$ (see for example the appendix to \cite{herzigthesis}).

If $a\in(\Z^n_+)_0^{\coprod_{w|l}\Hom(k_w,\Flbar)}$ is a Serre weight,
we define an irreducible $\Flbar$-representation $F_a$ of
$G(\cO_{F^+,l})$ as follows: we define \[F_a=\otimes_\Flbar
F_{a_\tv},\] an irreducible representation of
$\prod_{\tv\in\tilde{S}_l}\GL_n(k_v)$, and we let $G(\cO_{F^+,l})$ act on
$F_{a_\tv}$ by the composition of $\iota_\tv$ and reduction modulo
$l$. Again, we say that two Serre weights $a$ and $b$ are equivalent
if and only if $F_a\cong F_b$ as representations of
$G(\cO_{F^+,l})$. This is equivalent to demanding that for each place
$w|l$ and each $\sigma\in\Hom(k_w,\Flbar)$ and each $1\le i\le n-1$ we have  \[a_{\sigma,i}-a_{\sigma,i+1}=b_{\sigma,i}-b_{\sigma,i+1},\] and the
character $k_w^\times\to\Flbar^\times$ given
by \[x\mapsto\prod_{\sigma\in\Hom(k_w,\Flbar)}\sigma(x)^{a_{\sigma,n}-b_{\sigma,n}}\]is
trivial.

Note that the representation $F_a$ is independent of the choice
of $\tS_l$ (this follows easily from the condition that $a_{\sigma
  c,n+1-i}=-a_{\sigma,i}$ and the relation $\iota_{w^c}(x)={}^{t}(\iota_w(x))^{-1}$).

For future use, if $a\in(\Z^n_+)_0^{\coprod_{w|l}\Hom(k_w,\Flbar)}$ is
a Serre weight, we also define an $\Flbar$-representation $P_a$ of
$G(\cO_{F^+,l})$ as follows: we define \[P_a=\otimes_\Flbar
P_{a_\tv},\] a representation of
$\prod_{\tv\in\tilde{S}_l}\GL_n(k_v)$, and we let $G(\cO_{F^+,l})$ act on
$P_{a_\tv}$ by the composition of $\iota_\tv$ and reduction modulo
$l$. Note that $F_a$ is a subrepresentation of $P_a$.

We say that a weight $\lambda\in (\Z^n_+)_0^{\Hom(F,\Qlbar)}$ is a
\emph{lift} of a Serre weight $a$ if for each $w|l$ and each
$\sigma\in\Hom(k_w,\Flbar)$ there is an element
$\tau\in\Hom(F,\Qlbar)$ lying over $w$ and lifting $\sigma$ such
that $\lambda_{\tau}=a_\sigma$, and for all other
$\tau'\in\Hom(F,\Qlbar)$ lying over $w$ and lifting $\sigma$ we have
$\lambda_{\tau'}=0$. If $\lambda\in (\Z^n_+)_0^{\Hom(F,\Qlbar)}$ and
$w|l$ is a place of $F$, we let $\lambda_w\in
(\Z^n_+)^{\Hom(F_w,\Qlbar)}$ be defined in the obvious way. If $L$ is
a finite extension of $\Ql$ with residue field $k_L$, we say that an
element $\lambda\in(\Z^n_+)^{\Hom(L,\Qlbar)}$ is a \emph{lift} of an
element $a\in(\Z^n_+)^{\Hom(k_L,\Flbar)}$ if for each $\sigma\in\Hom(k_L,\Flbar)$ there is an element
$\tau\in\Hom(L,\Qlbar)$ lifting $\sigma$ such
that $\lambda_{\tau}=a_\sigma$, and for all other
$\tau'\in\Hom(L,\Qlbar)$ lifting $\sigma$ we have
$\lambda_{\tau'}=0$.

For the rest of this section, fix $K=\Qlbar$.
\begin{defn}\label{defn: good subgroup} We say that a compact open subgroup
of $G(\A_{F^+}^\infty)$ is \emph{good} if $U=\prod_vU_v$ with $U_v$ a
compact open subgroup of $G(F^+_v)$ such that:
\begin{itemize}
\item $U_v\subset G(\bigO_{F^+_v})$ for all $v$ which split in $F$;
  \item $U_v=G(\bigO_{F^+_v})$ if $v|l$;
  \item $U_v$ is a hyperspecial maximal compact subgroup of $G(F_v^+)$
    if $v$ is inert in $F$.
\end{itemize}
\end{defn}

Let $U$ be a good compact open subgroup of $G(\A_{F^+}^\infty)$. Let $T$ be a finite set of finite places of $F^+$ which split in $F$,
containing $S_l$ and all the places $v$ which split in $F$ for which
$U_v\neq G(\bigO_{F^+_v})$. We let $\mathbb{T}^{T,\univ}$ be the
commutative $\bigO$-polynomial algebra generated by formal variables
$T_w^{(j)}$ for all $1\le j\le n$, $w$ a place of $F$ lying over a
place $v$ of $F^+$ which splits in $F$ and is not contained in $T$.
For any $\lambda\in (\Z^n_+)^{\tI_l}$, the algebra
$\mathbb{T}^{T,\univ}$ acts on $S_\lambda(U,\cO)$ via the
 Hecke operators
  \[ T_{w}^{(j)}:=  \iota_{w}^{-1} \left[ \GL_n(\mc{O}_{F_w}) \left( \begin{matrix}
      \varpi_{w}1_j & 0 \cr 0 & 1_{n-j} \end{matrix} \right)
\GL_n(\mc{O}_{F_w}) \right] 
\] for $w\not \in T$ and $\varpi_w$ a uniformiser in $\mc{O}_{F_w}$.  
Similarly, for any Serre weight
$a\in(\Z^n_+)_0^{\coprod_{v|l}\Hom(k_v,\Flbar)}$, $\mathbb{T}^{T,\univ}$
acts on $S(U,F_a)$.

Suppose that $\mf{m}$ is a maximal ideal of
$\mathbb{T}^{T,\univ}$ with residue field $\Flbar$ such that
$S_\lambda(U,\Qlbar)_{\mf{m}}\neq 0$. Then (cf.\ Proposition 3.4.2 of \cite{cht}) by Lemma \ref{lem: relationship of algebraic automorphic forms to classical
    automorphic forms}, Theorem \ref{thm: base change from G to GL},
  and Theorem \ref{thm: existence of Galois reps attached to RACSDC}, 
there is a continuous semisimple
representation \[\rbar_{\mf{m}}:G_{F}\to\GL_n(\Flbar)\]associated
to $\mf{m}$, which is uniquely determined by the properties that:
\begin{itemize}
\item $\rbar_{\mf{m}}^c\cong\rbar_{\mf{m}}^\vee\epsilonbar_l^{1-n}$,
\item for all finite
places $w$ of $F$ not lying over $T$, $\rbar_{\mf{\m}}|_{G_{F_w}}$ is
unramified, and
\item if $w$ is a finite place of $F$ which doesn't lie over $T$ and which splits over $F^+$,
  then the characteristic polynomial of  $\rbar_{\mf{\m}}(\Frob_w)$
  is \[X^n-T_w^{(1)}X^{n-1}+\dots+(-1)^j(\mathbf{N}w)^{j(j-1)/2}T_w^{(j)}X^{n-j}+\dots+(-1)^n(\mathbf{N}w)^{n(n-1)/2}T_w^{(n)}. \]
\end{itemize}

\begin{lem}\label{Lemma: equivalence of modular in char 0 and l of
    some weight, unitary group version}
  Suppose that $U$ is sufficiently small, and let $\mf{m}$ be a maximal ideal of
$\mathbb{T}_\lambda^{T,\univ}$ with residue field $\Flbar$. Suppose
that $a\in(\Z^n_+)_0^{\coprod_{v|l}\Hom(k_v,\Flbar)}$ is a Serre weight,
and that $\lambda\in(\Z^n_+)^{\tI_l}$ is a lift of $a$. Then
\[S_\lambda(U,\Qlbar)_{\mf{m}}\neq 0\] if and only if for some
Jordan-H\"older factor $F$ of the $G(\cO_{F^+,l})$-representation
$P_a$, \[S(U,F)_{\mf{m}}\neq 0.\] In particular if
$S(U,F_a)_{\mf{m}}\neq 0$ then $S_\lambda(U,\Qlbar)_{\mf{m}}\neq 0$.
\end{lem}
\begin{proof}
We have
$S_\lambda(U,\Qlbar)_{\mf{m}}=S_\lambda(U,\cO_\Qlbar)_{\mf{m}}\otimes\Qlbar$. Since
$U$ is sufficiently small, it follows that $S_\lambda(U,\cO_\Qlbar)_{\mf{m}}$ is
$l$-torsion free. Thus $S_\lambda(U,\Qlbar)_{\mf{m}}\neq
  0$ if and only if  $S_\lambda(U,\cO_\Qlbar)_{\mf{m}}\neq
  0$. However, using the fact that $U$ is sufficiently small again, we
  have $S_\lambda(U,\Flbar)_{\mf{m}}\neq
  0$ if and only if  $S_\lambda(U,\cO_\Qlbar)_{\mf{m}}\neq
  0$. Thus, $S_\lambda(U,\Qlbar)_{\mf{m}}\neq
  0$ if and only if  $S_\lambda(U,\Flbar)_{\mf{m}}\neq
  0$.

  But
  $S_\lambda(U,\Flbar)_{\mf{m}}=S(U,W_\lambda\otimes_\cO\Flbar)_\mf{m}$ is
  nonzero if and only if $S(U,F)_\mf{m}$ is nonzero for some
  Jordan-H\"older factor $F$ of $W_\lambda\otimes_\cO\Flbar$. (This
  follows from the exactness of the functor $F\mapsto S(U,F)_{\mf{m}}$
  which in turn follows from the fact that $U$ is sufficiently small.)
 It then suffices
  to note that as an immediate consequence of the definitions, we have
  $P_a \cong W_\lambda\otimes_{\cO} \Flbar$ and $F_a$
  is a Jordan-H\"older factor of $W_\lambda\otimes_\cO\F$.\end{proof}
We have the following definitions.

\begin{defn}\label{defn: galois split ramification}
   If $R$
  is a commutative ring and
  $r:G_F\to\GL_n(R)$ is a representation, we say that $r$ has
  \emph{split ramification} if $r|_{G_{F_w}}$ is unramified for any
  finite place $w\in F$ which does not split over $F^+$.
\end{defn}

\begin{defn}
  If $\pi$ is a RACSDC automorphic representation of $\GL_n(\A_F)$, we
  say that $\pi$ has \emph{split ramification} if $\pi_w$ is unramified for any
  finite place $w\in F$ which does not split over $F^+$.
\end{defn}

\begin{defn}\label{defn: modular of some Serre weight}
  Suppose that $\rbar:G_F\to\GL_n(\Flbar)$ is a continuous
  irreducible representation. Then we say that $\rbar$ \emph{is modular
  of weight} $a\in(\Z^n_+)_0^{\coprod_{w|l}\Hom(k_w,\Flbar)}$ if
  there is a good, sufficiently small level $U$, a set of places $T$ as
  above, and a maximal ideal $\mathfrak{m}$ of
  $\mathbb{T}^{T,\univ}$ with residue field $\Flbar$ such
  that
  \begin{itemize}
  \item $S(U,F_a)_{\mf{m}}\neq 0$, and
  \item $\rbar\cong \rbar_{\mathfrak{m}}$.
  \end{itemize} (Note that $\rbar_{\mf{m}}$ exists by Lemma \ref{Lemma: equivalence of modular in char 0 and l of
    some weight, unitary group version} and the remarks preceding it.) We say that $\rbar$ is modular if it is modular of
  some weight.
\end{defn}
\begin{remark}
  Note that if $\rbar:G_F\to\GL_n(\Flbar)$ is modular then $\rbar$
  must have split ramification, and
  $\rbar^c\cong\rbar^\vee\epsilonbar_l^{1-n}$. Note also that this
  definition is independent of the choice of $\tilde{S}_l$ (because
  $F_a$ is independent of this choice).\end{remark}

\begin{lem}\label{lem: equivalence of modular of Serre weight and
    RACSDC lift}Suppose that $\rbar:G_F\to\GL_n(\Flbar)$ is a continuous
  irreducible representation with split ramification. Let
  $a\in(\Z^n_+)_0^{\coprod_{w|l}\Hom(k_w,\Flbar)}$ be a Serre weight, and
  let $\lambda\in(\Z^n_+)_0^{\Hom(F,\Qlbar)}$ be a lift of $a$. Then
  if $\rbar$ is modular
  of weight $a\in(\Z^n_+)_0^{\coprod_{w|l}\Hom(k_w,\Flbar)}$, there is a RACSDC automorphic representation $\pi$ of
  $\GL_n(\A_F)$ of weight $\imath\lambda$ and level prime to
  $l$ which has split ramification, and which satisfies
  $\rbar_{l,\imath}(\pi)\cong\rbar$. Conversely, if there is a RACSDC automorphic representation $\pi$ of
  $\GL_n(\A_F)$ of weight $\imath\lambda$ and level prime to
  $l$ which has split ramification, and which satisfies
  $\rbar_{l,\imath}(\pi)\cong\rbar$, then $\rbar$ is modular of weight
  $b\in(\Z^n_+)_0^{\coprod_{w|l}\Hom(k_w,\Flbar)}$ for some $b$ such that
  the $G(\cO_{F^+,l})$-representation $P_a$ has a Jordan-H\"older factor
  isomorphic to $F_b$.
\end{lem}
\begin{proof}Suppose firstly that $\rbar$ is modular of weight
  $a$. Then by definition there is a good $U$ and a $T$ as
  above with $U$ sufficiently small, and a maximal ideal $\mathfrak{m}$ of
  $\mathbb{T}^{T,\univ}$ with residue field $\Flbar$ such
  that
  \begin{itemize}
  \item $S(U,F_a)_{\mf{m}}\neq 0$, and
  \item $\rbar\cong \rbar_{\mathfrak{m}}$.
  \end{itemize}By Lemma \ref{Lemma: equivalence of modular in char 0
    and l of some weight, unitary group version}, the first property implies
  that $S_\lambda(U,\Qlbar)_{\mf{m}}\neq 0$. Define a compact open
  subgroup $U'=\prod_w U_w'$ of $\GL_n(\A_F^\infty)$ by
  \begin{itemize}
  \item $U_w'=\GL_n(\cO_{F_w})$ if $w$ is not split over $F^+$.
  \item $U_w'=\iota_w(U_w)$ if $w$ splits over $F^+$.
  \end{itemize}By Lemma \ref{lem: relationship of algebraic automorphic forms to classical
    automorphic forms}, Theorem \ref{thm: base change from G to GL},
  and Theorem \ref{thm: existence of Galois reps attached to RACSDC},
  there is a RACSDC automorphic representation $\pi$ of weight
  $\lambda$  which satisfies $\rbar_{l,\iota}(\pi)\cong\rbar$, and
  $\pi_w^{U'_w}\ne 0$ for all finite places $w$ of $F$. Since $U$ is
  good, we see that $\pi$ has level prime to $l$, and it has split
  ramification, as required.

Conversely, suppose that there is a RACSDC automorphic representation $\pi$ of
  $\GL_n(\A_F)$ of weight $\lambda$ which has split ramification and level prime to $l$ with
  $\rbar_{l,\iota}(\pi)\cong\rbar$. Then there is a compact open
  subgroup $U'=\prod_wU'_w$ of $\GL_n(\A_F^\infty)$ such that
  \begin{itemize}
  \item for each finite place $w$ of $F$, $\pi_w^{U'_w}\ne 0$,
  \item $U'_w\subset\GL_n(\cO_{F_w})$ for all $w$,
  \item $U'_w=\GL_n(\cO_{F_w})$ for all $w|l$ and all $w$ which are
    not split over $F^+$,
  \item if $v=ww^c$ is a place of $F^+$ which splits in $F$, then
    $U_{w^c}=c({}^tU_w^{-1})$,
  \item there is a finite place $w$ of $F$ which is split over $F^+$ such
    that
    \begin{itemize}
    \item $w$ lies above a rational prime $p$ with $[F(\zeta_p):F]>n$, and
    \item $U'_w=\ker(\GL_n(\cO_w)\to\GL_n(\cO_w/\varpi_w))$.
    \end{itemize}

  \end{itemize}
  Define a compact open subgroup $U=\prod_v U_v$ of
  $G(\A_{F^+}^\infty)$ by
  \begin{itemize}
  \item if $v$ is inert in $F$, then $U_v=G(\cO_{F^+_v})$, and
  \item if $v=ww^c$ splits in $F$, then $U_v=\iota_w^{-1}(U'_w)$
    (which is well-defined by the fourth bullet point above).
  \end{itemize}
By the final bullet point in the list of properties of $U'$ above, $U$ is
sufficiently small. Then by Lemma \ref{lem: relationship of algebraic automorphic forms to classical
    automorphic forms} and Theorem \ref{thm: base
    change from GL to G} we have $S_\lambda(U,\Qlbar)_{\mf{m}}\ne
  0$. The result now follows from Lemma \ref{Lemma: equivalence of modular in char 0 and l of
    some weight, unitary group version}.
\end{proof}

\section{A lifting theorem}\label{sec:A lifting
  theorem}\subsection{}We recall some terminology from \cite{BLGGT},
specialized to the crystalline (as opposed to potentially crystalline)
case. Fix a prime
$l$. Let $K$ be a finite extension of $\Ql$, and $\cO$ the ring of
integers in a finite extension of $\Ql$ inside $\Qlbar$, with residue field
$k$. Assume that for each continuous embedding $K \into \barQQ_l$, the
image is contained in the field of fractions of $\cO$.

Let $\rhobar:G_K\to\GL_n(k)$ be a continuous representation, and let
$R_{\cO,\rhobar}^\Box$ be the universal $\cO$-lifting ring. Let $\{
H_\tau \}$ be a collection of $n$ element multisets of integers
parametrized by $\tau \in \Hom_{\Q_l}(K,\barQQ_l)$. Then
$R_{\CO,\barrho}^\Box$ has a unique quotient $R_{\CO,\barrho,\{
  H_\tau\}, \cris}^\Box$ which is reduced and without $l$-torsion and
such that a $\barQQ_l$-point of $R_{\CO,\barrho}^\Box$ factors through
$R_{\CO,\barrho,\{ H_\tau\}, \cris}^\Box$ if and only if it
corresponds to a representation $\rho:G_K \ra \GL_n(\barQQ_l)$ which is
crystalline and has $\HT_\tau(\rho)=H_\tau$ for all $\tau:K \into
\barQQ_l$. We will write $R_{\barrho,\{ H_\tau\}, \cris}^\Box \otimes
\barQQ_l$
for  $R_{\CO,\barrho,\{ H_\tau\}, \cris}^\Box \otimes_{\CO} \barQQ_l$. 
This definition is independent of the choice of $\CO$. The scheme
$\Spec (R_{\barrho,\{ H_\tau\}, \cris}^\Box \otimes \barQQ_l)$ is
formally smooth over $\Spec \Qlbar$. (See \cite{kisindefrings}.)

Let $\rho_1,\rho_2 : G_K \to \GL_n(\CO_{\Qlbar})$ be continuous
representations. We say that $\rho_1$ {\em connects to} $\rho_2$, which we denote $\rho_1
\sim \rho_2$, if and only if 
\begin{itemize}
\item the reduction $\barrho_1=\rho_1 \bmod \gm_{\CO_{\barQQ_l}}$ is equivalent to the reduction $\barrho_2= \rho_2 \bmod \gm_{\barQQ_l}$;
\item $\rho_1$ and $\rho_2$ are both  crystalline;
\item for each $\tau:K \into \barQQ_l$ we have $\HT_\tau(\rho_1)=\HT_\tau(\rho_2)$;
\item and $\rho_1$ and $\rho_2$ define points on the same irreducible component of the scheme $\Spec (R_{\barrho_1,\{ \HT_\tau(\rho_1)\}, \cris}^\Box \otimes \barQQ_l)$.
\end{itemize}
We note that $\rho_1 \sim \rho_2$ in our sense if and only if both
$\rho_1$ and $\rho_2$ are crystalline and $\rho_1\sim \rho_2$ in the
sense of \cite{BLGGT}. As in section 2.3 of \cite{BLGGT}, we have the following:
\begin{enumerate}
\item The relation $\rho_1 \sim \rho_2$ does not depend 
on the equivalence chosen between the reductions $\barrho_1$ and $\barrho_2$, nor on the $\GL_n(\CO_{\barQQ_l})$-conjugacy class of $\rho_1$ or $\rho_2$.
\item $\sim$ is symmetric and transitive. 
\item If $K'/K$ is a finite extension and $\rho_1 \sim \rho_2$ then $\rho_1|_{G_{K'}} \sim \rho_2|_{G_{K'}}$. 
\item If $\rho_1 \sim \rho_2$ and $\rho_1' \sim \rho_2'$ then
  $\rho_1\oplus \rho_1'  \sim \rho_2 \oplus \rho_2'$ and
  $\rho_1\otimes \rho_1'  \sim \rho_2 \otimes \rho_2'$ and $\rho_1^\vee\sim\rho_2^\vee$.
\item If $\mu:G_K \ra \barQQ_l^\times$ is a continuous unramified
  character with $\mubar=1$ then $\rho_1 \sim \rho_1 \otimes \mu$.
\item \label{semisimp}  Suppose $\rho_1$ is crystalline and
  $\barrho_1$ is semisimple. Let $\Fil^i$ be a $G_K$-invariant filtration on
  $\rho_1$ by $\CO_{\Qlbar}$-direct summands. Then $\rho_1 \sim
  \oplus_i \gr^i(\Fil)$.
\end{enumerate}

We will call a crystalline representation $\rho:G_K \ra \GL_n(\CO_{\barQQ_l})$ {\em diagonal} if 
it is of the form $\chi_1 \oplus \dots \oplus \chi_n$ with $\chi_i:G_K \ra \CO_{\barQQ_l}^\times$.
We will call a crystalline representation $\rho:G_K \ra \GL_n(\CO_{\barQQ_l})$ {\em diagonalizable} if 
it connects to some diagonal representation. We will call a representation $\rho_1:G_K \ra \GL_n(\CO_{\barQQ_l})$ {\em potentially diagonalizable} if there is a finite extension $K'/K$ such that $\rho_1|_{G_{K'}}$ is diagonalizable.
Note that if $K''/K$ is a finite extension and $\rho_1$ is diagonalizable (resp. potentially diagonalizable)
then $\rho_1|_{G_{K''}}$ is diagonalizable (resp. potentially
diagonalizable).

Suppose now that $K$ is a finite extension of $\Q_p$ for some prime $p \neq l$
and
\[ \rho_1,\rho_2 : G_K \to \GL_n(\CO_{\Qlbar}) \]
are two continuous representations. We define the notion that $\rho_1$
 connects to $\rho_2$ exactly as in \cite{BLGGT}. Again, this will
be denoted by $\rho_1 \sim \rho_2$.

Recall the following definition from \cite{jack} (for a discussion of
the equivalence of this definition to that formulated in \cite{jack},
see the appendix to \cite{BLGGU2}). \begin{defn}\label{defn:adequate} We call a finite
subgroup $H \subset \GL_n(\barFF_l)$ {\em adequate} if the following
conditions are satisfied.
\begin{enumerate}
\item $H$ has no non-trivial quotient of $l$-power order (i.e.\ $H^1(H,\barFF_l)=(0)$).
\item $l \ndiv n$.
\item The elements of $H$ with order coprime to $l$ span $M_{n \times
    n}(\barFF_l)$ over $\barFF_l$. (This implies that $\barFF_l^n$ is an irreducible representation of $H$.)
\item $H^1(H,\gothgl_n(\barFF_l))=(0)$.
\end{enumerate}
\end{defn}

In particular, we have the following useful result, an immediate
consequence of Theorem 9 of \cite{jackapp}.
\begin{thm}
  Suppose that $l\ge 2(n+1)$, and that $H$ is a finite subgroup of
  $\GL_n(\Flbar)$ which acts irreducibly. Then $H$ is adequate.
\end{thm}

Fix an isomorphism $\iota:\Qlbar\to\C$. Let $F$ be an imaginary CM field with maximal totally real subfield
$F^+$. 

\begin{thm}\label{thm: existence of lifts, pot diag
    components}
  Let $l>2$ be prime, and let $F$ be a CM field with maximal totally
  real subfield $F^+$, with $\zeta_l\notin F$. Assume that the
  extension $F/F^+$ is split at all places dividing $l$.
Suppose that \[\rbar:G_F\to\GL_n(\Flbar)\] is an irreducible
representation  which satisfies the following properties.
\begin{enumerate}
\item There is a RACSDC automorphic representation $\Pi$ of
  $\GL_n(\A_F)$ such that
  \begin{itemize}
  \item $\rbar\cong\rbar_{l,\imath}(\Pi)$ (so in particular, $\rbar^c\cong\rbar^\vee\epsilonbar_l^{1-n}$).
  \item For each place $w|l$ of $F$, $r_{l,\imath}(\Pi)|_{G_{F_w}}$ is
  potentially diagonalizable.
  \end{itemize}
\item The image $\rbar(G_{F(\zeta_l)})$ is adequate.
\end{enumerate}
Let $S$ be a finite set of finite places of $F^+$ which split in
$F$. Assume that $S$ contains all the places of $F^+$ dividing $l$,
and all places lying under a place of $F$ at which $\rbar$ is
ramified. For each $v\in S$ choose a place $\tv$ of $F$ above $v$, and
a lift $\rho_{\tv}:G_{F_\tv}\to\GL_n(\CO_{\Qlbar})$ of
$\rbar|_{G_{F_\tv}}$. Assume that if $v|l$, then $\rho_\tv$ is
crystalline and potentially diagonalizable, and if
$\tau:F_\tv\into\Qlbar$ is any embedding, then $\HT_\tau(\rho_\tv)$
consists of distinct integers.

Then there is a RACSDC automorphic representation $\pi$ of
$\GL_n(\A_F)$ of level prime to $l$ such that
\begin{itemize}
\item $\rbar\cong\rbar_{l,\iota}(\pi)$.
\item $\pi_w$ is unramified for all $w$ not lying over a place of $S$,
  so that $r_{l,\iota}(\pi_w)$ is unramified at all such $w$.
\item $r_{l,\iota}(\pi)|_{G_{F_\tv}}\sim\rho_\tv$ for all $v\in S$. In
  particular, for each place $v|l$,  $r_{l,\iota}(\pi)|_{G_{F_\tv}}$ is
  crystalline and for each embedding $\tau:F_\tv\into\Qlbar$,
  $\HT_\tau(r_{l,\iota}(\pi)|_{G_{F_\tv}})=\HT_\tau(\rho_\tv)$.
\end{itemize}

\end{thm}
\begin{proof}Let $\cG_n$ be the group scheme over $\Z$ defined in
  section 2.1 of \cite{cht}. Then by the main result of
  \cite{belchen}, $\rbar$ extends to a representation
  $\rhobar:G_{F^+}\to\cG_n(\Flbar)$ with multiplier
  $\epsilonbar_l^{1-n}$.

We now apply Theorem A.4.1 of \cite{BLGGU2}, with
\begin{itemize}
\item $F$, $n$ and $S$ as in the present setting.
\item $\rbar$ our present $\rhobar$.
\item $\rho_v$ our $\rho_\tv$.
\item $\mu=\epsilon_l^{1-n}$.
\item $F'=F$.
\end{itemize}
We conclude that $\rbar$ has a lift $r:G_F\to\GL_n(\Qlbar)$ (the
restriction to $G_F$ of the representation $r$ of Theorem A.4.1 of \cite{BLGGU2}) such that
\begin{itemize}
\item $r^c\cong r^\vee\epsilon_l^{1-n}$.
\item if $v\in S$ then $r|_{G_{F_\tv}}\sim \rho_\tv$.
\item $r$ is unramified outside $S$.
\item $r$ is automorphic of level potentially prime to $l$.
\end{itemize}
 By Theorem \ref{thm: existence of
  Galois reps attached to RACSDC}, we see that (since $r|_{G_{F_w}}$
is crystalline for all $w|l$, and unramified at all places $w$ not
lying over a place in $S$)
$\pi_w$ is unramified for all $w|l$ and all $w$ not lying over a place
in $S$, as required.
\end{proof}

\section{Serre weight
  conjectures}\label{sec:Conjectures}\subsection{}We now briefly
discuss Serre weight conjectures for $\GL_n$. We refer the reader to
the forthcoming \cite{GHS} for a far more detailed discussion. In
particular, in much of this section we restrict ourselves to the case
that $l$ splits completely in $F$, both for simplicity of notation and
because in this case we can prove theorems with cleaner conditions, as
representations satisfying the Fontaine-Laffaille condition are always
potentially diagonalizable.

Let $K$ be a finite extension of $\Ql$, with ring of integers $\cO_K$
and residue field $k$. Let $\rhobar:G_K\to\GL_n(\Flbar)$ be a
continuous representation. Then it is a folklore conjecture that for each such
$\rhobar$, there
is a set $W(\rhobar)$ of Serre weights of $\GL_n(k)$ for each $K$ and
each $\rhobar$ with the following property:
if $F$ is a CM field, $\rbar:G_F\to\GL_n(\Flbar)$ is an
irreducible modular representation (so in particular it is conjugate
self-dual), $w|l$ is a place of $F$ and $\sigma_w$ is an
irreducible $\Flbar$-representation of $\GL_n(k_w)$, then $\rbar$ is modular of
Serre weight $\sigma_w\otimes_{\Flbar}\sigma^w$ for some $\sigma^w$ if and only
if $\sigma_w\in W(\rbar|_{G_{F_w}})$.

It is natural to believe that there is a description of $W(\rhobar)$ in terms of the
existence of crystalline lifts with particular Hodge-Tate weights, as
we now explain. This is one of the motivations for the general Serre
weight conjectures explained in \cite{GHS}.

\begin{defn}
  \label{defn: Galois representation of Hodge type some weight}Let
  $K/\Ql$ be a finite extension, let
  $\lambda\in(\Z^n_+)^{\Hom(K,\Qlbar)}$, and let
  $\rho:G_K\to\GL_n(\Qlbar)$ be a de Rham representation. Then we say
  that $\rho$ has \emph{Hodge type} $\lambda$ if for each
  $\tau\in\Hom(K,\Qlbar)$, we have $\HT_\tau(\rho)=\{\lambda_{\tau,1}+(n-1),\lambda_{\tau,2}+(n-2),\dots,\lambda_{\tau,n}\}$.
\end{defn}
\begin{rem}\label{rem: hodge type of auto Galois rep}
  As an immediate consequence of this definition and of Theorem
  \ref{thm: existence of Galois reps attached to RACSDC}, we see that
  if $\pi$ is a RACSDC automorphic representation of weight
  $\lambda\in(\Z^n_+)_0^{\Hom(F,\C)}$, then for each place $w|l$,
  $r_{l,\imath}(\pi)|_{G_{F_w}}$ has Hodge type $(\imath^{-1}\lambda)_w$.
\end{rem}

\begin{lemma}\label{lem:modularofaweightimpliescrystallinelifts -
    before the definition of the set of weights}
  Let $n$ be a positive integer, and let $F$ be an imaginary CM field
  with maximal totally real subfield $F^+$, and suppose that $F/F^+$
  is unramified at all finite places, that every place of $F^+$
  dividing $l$ splits completely in $F$, and that if $n$ is even then
  $n[F^+:\Q]/2$ is even. Suppose that $\rbar:G_F\to\GL_n(\Flbar)$ is
  an irreducible modular representation with split ramification. Let
  $a\in(\Z^n_+)_0^{\coprod_{w|l}\Hom(k_w,\Flbar)}$ be a Serre weight,
  and let $\lambda\in(\Z^n_+)_0^{\Hom(F,\Qlbar)}$ be a lift of $a$. If
  $\rbar$ is modular of weight $a$, then for each place $w|l$ there is
  a continuous lift $r_w:G_{F_w}\to\GL_n(\CO_{\Qlbar})$ of
  $\rbar|_{G_{F_w}}$ such that $r_w$ is crystalline of Hodge type
  $\lambda_w$.

\end{lemma}
\begin{proof}
  By Lemma \ref{lem: equivalence of modular of Serre weight and
    RACSDC lift} there is a RACSDC automorphic representation $\pi$ of
  $\GL_n(\A_F)$, which has level prime to $l$ and weight
  $\imath\lambda$, such that $\rbar_{l,\imath}(\pi)\cong\rbar$. Then
  we may take $r_w:=r_{l,\imath}(\pi)|_{G_{F_w}}$, which has the
  required properties by Remark \ref{rem: hodge type of auto Galois
    rep}.
\end{proof}

This suggests the following definition.
\begin{defn}Let $K$ be a finite extension of $\Ql$, with ring of
  integers $\cO_K$ and residue field $k$. Let
  $\rhobar:G_K\to\GL_n(\Flbar)$ be a continuous representation. Then
  we let $\Wcris(\rhobar)$ be the set of Serre weights
  $a\in(\Z_+^n)^{\Hom(k,\Flbar)}$ with the property that there is a
  crystalline representation
  $\rho:G_K\to\GL_n(\CO_{\Qlbar})$ lifting $\rhobar$, such that $\rho$ has Hodge type $\lambda$ for some lift
    $\lambda\in(\Z^n_+)^{\Hom(K,\Qlbar)}$ of $a$.

\end{defn}

The results of section \ref{sec:A lifting theorem} suggest the
following definition.

\begin{defn}Let $K$ be a finite extension of $\Ql$, with ring of
  integers $\cO_K$ and residue field $k$. Let
  $\rhobar:G_K\to\GL_n(\Flbar)$ be a continuous representation. Then
  we let $\Wdiag(\rhobar)$ be the set of Serre weights
  $a\in(\Z_+^n)^{\Hom(k,\Flbar)}$ with the property that there is a
  potentially diagonalizable crystalline representation
  $\rho:G_K\to\GL_n(\Qlbar)$ lifting $\rhobar$, such that
  $\rho$ has Hodge type $\lambda$ for some lift
    $\lambda\in(\Z^n_+)^{\Hom(K,\Qlbar)}$ of $a$.

\end{defn}
\begin{remark}
  If $a$ and $b$ are
  equivalent Serre weights, then $a\in\Wcris(\rhobar)$ (respectively
  $\Wdiag(\rhobar)$) if and only if $b\in\Wcris(\rhobar)$
  (respectively $\Wdiag(\rhobar)$). This is an easy consequence of
  Lemma 4.1.15 of \cite{BLGGU2}, which provides
  a crystalline character with trivial reduction with which one can twist
  the crystalline Galois representations of Hodge type some lift of
  $a$ to obtain crystalline representations of Hodge type some lift of
  $b$. The same remarks apply to the set $\Wexplicit(\rhobar)$ defined
  below.
\end{remark}

By definition we have $\Wdiag(\rhobar)\subset \Wcris(\rhobar)$. We
``globalise'' these definitions in the obvious way:
\begin{defn}\label{defn: serre weights global}
  Let $\rbar:G_F\to\GL_n(\Flbar)$ be a continuous
  representation with $\rbar^c\cong\rbar^\vee\epsilonbar_l^{1-n}$. Then we let $\Wcris(\rbar)$ (respectively
  $\Wdiag(\rbar)$) be the set of Serre weights
  $a\in(\Z^n_+)_0^{\coprod_{w|l}\Hom(k_w,\Flbar)}$ such that for each
  place $w|l$, the corresponding Serre weight
  $a_w\in(\Z^n_+)^{\Hom(k_w,\Flbar)}$ is an element of
  $\Wcris(\rbar|_{G_{F_w}})$ (respectively $\Wdiag(\rbar|_{G_{F_w}})$).
  
\end{defn}
The
point of these definitions is the following Corollary and Theorem.

\begin{cor}
  \label{cor: modular of some weight implies crystalline lifts
    exist}Let $n$ be a positive integer, let $F$ be an imaginary CM field with maximal totally real subfield
  $F^+$, and suppose that $F/F^+$ is unramified at all finite places,
  that every place of $F^+$ dividing $l$ splits completely in $F$,
  and that if $n$ is even then $n[F^+:\Q]/2$ is even. Suppose that
  $\rbar:G_F\to\GL_n(\Flbar)$ is an irreducible modular
  representation with split ramification. Let $a\in(\Z^n_+)_0^{\coprod_{w|l}\Hom(k_v,\Flbar)}$ be a
  Serre weight. If $\rbar$ is modular of weight $a$, then $a\in \Wcris(\rbar)$.
\end{cor}
\begin{proof}
  This is an immediate consequence of Lemma \ref{lem:modularofaweightimpliescrystallinelifts -
    before the definition of the set of weights} and Definition \ref{defn: serre weights global}.
\end{proof}
\begin{thm}
  \label{thm: potentially diagonalizable local lifts implies Serre
    weight}Let $F$ be an imaginary CM field with maximal totally real subfield
  $F^+$, and suppose that $F/F^+$ is unramified at all finite places,
  that every place of $F^+$ dividing $l$ splits completely in $F$,
  and that  if $n$ is even then $n[F^+:\Q]/2$ is even. Assume that
  $\zeta_l\notin F$. 
  Suppose that $l>2$, and that
  $\rbar:G_F\to\GL_n(\Flbar)$ is an irreducible 
  representation with split ramification. Assume that

\begin{itemize}
\item There is a RACSDC automorphic representation $\Pi$ of
  $\GL_n(\A_F)$ such that
  \begin{itemize}
  \item $\rbar\cong\rbar_{l,\imath}(\Pi)$ (so in particular, $\rbar^c\cong\rbar^\vee\epsilonbar_l^{1-n}$).
  \item For each place $w|l$ of $F$, $r_{l,\imath}(\Pi)|_{G_{F_w}}$ is
  potentially diagonalizable.
   \item $\rbar(G_{F(\zeta_l)})$ is adequate.
  \end{itemize}
\end{itemize}
 Let $a\in(\Z^n_+)_0^{\coprod_{w|l}\Hom(k_w,\Flbar)}$ be a
  Serre weight. Assume that $a\in \Wdiag(\rbar)$. 
Then there is a Serre weight
$b\in(\Z^n_+)_0^{\coprod_{w|l}\Hom(k_w,\Flbar)}$ such that
\begin{itemize}
\item $\rbar$ is modular of weight $b$.
\item There is a Jordan-H\"older factor of the $G(\cO_{F^+,l})$
  representation $P_a$ which is isomorphic to $F_b$.

\end{itemize}

\end{thm}
\begin{proof}By the assumption that $a\in \Wdiag(\rbar)$, there is a
  lift $\lambda$ of $a$ such that for each $w|l$ there is a
  potentially diagonalizable crystalline lift
  $\rho_w:G_{F_w}\to\GL_n(\CO_{\Qlbar})$ of $\rbar|_{G_{F_w}}$ of
  Hodge type $\lambda_w$.  

By Theorem \ref{thm: existence of lifts,
    pot diag components}, there is a RACSDC automorphic representation
  $\pi$ of $\GL_n(\A_F)$ of weight $\imath\lambda$, of level prime to
  $l$ and with split ramification, such that
  $\rbar(\pi)\cong\rbar$. The result follows from Lemma \ref{lem:
    equivalence of modular of Serre weight and RACSDC
    lift}.\end{proof} Since Fontaine--Laffaille representations are
potentially diagonalizable, we obtain the following Corollary.
\begin{cor}
  \label{cor: Fontaine-Laffaille potentially diagonalizable local lifts implies Serre
    weight}Let $F$ be an imaginary CM field with maximal totally real subfield
  $F^+$, and suppose that $F/F^+$ is unramified at all finite places,
  that every place of $F^+$ dividing $l$ splits completely in $F$,
  and that  if $n$ is even then $n[F^+:\Q]/2$ is even. Suppose that $l>2$, and that
  $\rbar:G_F\to\GL_n(\Flbar)$ is an irreducible 
  representation with split ramification. Assume that

\begin{enumerate}
\item $l$ is unramified in $F$.
\item There is a RACSDC automorphic representation $\Pi$ of
  $\GL_n(\A_F)$ of weight $\mu\in(\Z^n_+)_0^{\Hom(F,\C)}$ and level
  prime to $l$ such that
  \begin{itemize}
  \item $\rbar\cong\rbar_{l,\imath}(\Pi)$ (so in particular, $\rbar^c\cong\rbar^\vee\epsilonbar_l^{1-n}$).
  \item For each $\tau\in\Hom(F,\C)$, $\mu_{\tau,1}-\mu_{\tau,n}\le l-n$.
    \item $\rbar(G_{F(\zeta_l)})$ is adequate.
  \end{itemize}
\end{enumerate}
 Let $a\in(\Z^n_+)_0^{\coprod_{w|l}\Hom(k_w,\Flbar)}$ be a
  Serre weight. Assume that $a\in \Wdiag(\rbar)$. 
Then there is a Serre weight
$b\in(\Z^n_+)_0^{\coprod_{w|l}\Hom(k_w,\Flbar)}$ such that
\begin{itemize}
\item $\rbar$ is modular of weight $b$.
\item There is a Jordan-H\"older factor of the $G(\cO_{F^+,l})$
  representation $P_a$ which is isomorphic to $F_b$.

\end{itemize}

\end{cor}
\begin{proof}
  By Theorem \ref{thm: potentially diagonalizable local lifts implies
    Serre weight}, it is enough to check that for each place $w|l$ of
  $F$, $r_{l,\imath}(\Pi)|_{G_{F_w}}$ is potentially diagonalizable.
  This follows from the main result of~\cite{GaoLiu12}.
\end{proof}

As explained above, we now specialise to the case that $l$ splits
completely in $F$. We further assume that $\rbar|_{G_{F_w}}$ is
semisimple for all $w|l$, and specify a set $\Wexplicit(\rbar)$ of
Serre weights. These weights will have the property that if
$a\in\Wexplicit(\rbar)$, and $\lambda$ is the unique lift of $a$ to
$(\Z^n_+)_0^{\Hom(F,\Qlbar)}$, then for each place $w|l$,
$\rbar|_{G_{F_w}}$ has a potentially diagonalizable (indeed
potentially diagonal) crystalline lift of Hodge type $\lambda_w$.

Since the situation is purely local, we change notation and work with
$G_\Ql$. Let $\Q_{l^m}$ denote the unramified extension of $\Ql$ of
degree $m$ inside $\Qlbar$, and let
$\omega_m:G_{\Q_{l^m}}\to\Flbar^\times$ denote a choice of fundamental
character of niveau $m$ (this is given by the action of $G_{\Q_{l^m}}$
on the $(l^m-1)$-st roots of $l$). Given $\lambda\in\Flbar^\times$ and
an $m$-tuple of integers $\underline{c}=(c_0,\dots,c_{m-1})$, we
consider the representation
\[\rhobar_{\lambda,\underline{c}}:=\nr_\lambda\otimes\Ind_{G_{\Q_{l^m}}}^\GQl\omega_m^{-(c_0+lc_1+\dots+l^{m-1}c_{m-1})},\]
where $\nr_\lambda$ is the unramified character taking a geometric
Frobenius to $\lambda$. Given a partition $\underline{n}=n_1+\dots+n_r$, elements
$\underline{\lambda}=(\lambda_1,\dots,\lambda_r)$ of $\Flbar^\times$,
and a tuple $\underline{c}=(\underline{c}_1,\dots,\underline{c_r})$ of tuples
$\underline{c}_i=(c_{i,0},\dots,c_{i,n_i-1})$ of integers, 
 we
define the representation \[\rho_{\underline{n},\underline{\lambda},\underline{c}}:=\oplus_{i=1}^r\rho_{\lambda_i,\underline{c}_i}.\]
Note that we can we can think of $\underline{c}$ as the element $(c_{1,0},c_{1,2},\dots,c_{r,n_r-1})$ of
$\Z^n$, where $n = n_1+\dots+n_r$.

\begin{defn} Let $\rhobar:\GQl\to\GL_n(\Flbar)$ be a semisimple
  representation. Let $\Wexplicit(\rhobar)$ be the set of Serre
  weights $a\in\Z^n_+$ for which there exists a permutation $\sigma\in
  S_n$, a partition $\underline{n}$ of $n$ and $\underline{\lambda}$
  as above such
  that \[\rhobar\cong\rhobar_{\underline{n},\underline{\lambda},(a_{\sigma(1)}+n-\sigma(1),\dots,a_{\sigma(n)}+n-\sigma(n))}.\]
  We let $\Wexplicit_I(\rhobar)$ be the set of Serre weights $a$ for
  which there exist $\sigma$, $\underline{n}$ and
  $\underline{\lambda}$ such
  that \[\rhobar|_{I_\Ql}\cong\rhobar_{\underline{n},\underline{\lambda},(a_{\sigma(1)}+n-\sigma(1),\dots,a_{\sigma(n)}+n-\sigma(n))}|_{I_{\Ql}};\]this
  notion will useful in Section \ref{sec: Fontaine-Laffaille theory}.
\end{defn}
\begin{lemma}\label{lem: local explicit implies diag}
  If $\rhobar:\GQl\to\GL_n(\Flbar)$ is a semisimple representation and
  $a\in\Wexplicit(\rhobar)$, then $\rhobar$ has a potentially
  diagonalizable crystalline lift of Hodge type $a$.
\end{lemma}
\begin{proof}
  By the definition of ``Hodge type $a$'', it is enough to show that
  each representation $\rhobar_{\lambda,\underline{c}}:\GQl\to\GL_m(\Flbar)$
  defined above has a potentially diagonalizable crystalline lift with
  Hodge--Tate weights $c_0,\dots,c_{m-1}$ (note that the direct sum of
  potentially diagonalizable representations is again potentially
  diagonalizable). It thus suffices to show that the character
  $\omega_m^{-(c_0+lc_1+\dots+lc_{m-1})}$ of $G_{\Q_{l^m}}$ has a
  crystalline lift with Hodge--Tate weights $c_0,\dots,c_{m-1}$
  (because the induction to $\GQl$ of such a lift is certainly
  potentially diagonalizable). This follows at once from Lemma 6.2 of
  \cite{geesavitttotallyramified} (noting that the conventions on the
  sign of Hodge--Tate weights in \cite{geesavitttotallyramified} are
  the opposite of those of this paper).
\end{proof} Again we may globalise this definition in the obvious way.

\begin{defn}\label{defn: serre weights explicit global}
   Continue to assume that
    $l$ splits completely in $F$, and let $\rbar:G_F\to\GL_n(\Flbar)$ be a continuous
  representation with $\rbar^c\cong\rbar^\vee\epsilonbar_l^{1-n}$ and
  such that $\rbar|_{G_{F_w}}$ is semisimple for each $w|l$. Then we let $\Wexplicit(\rbar)$ be the set of Serre weights
  $a\in(\Z^n_+)_0^{\coprod_{w|l}\Hom(k_w,\Flbar)}$ such that for each
  place $w|l$, the corresponding Serre weight
  $a_w\in(\Z^n_+)^{\Hom(k_w,\Flbar)}$ is an element of
  $\Wexplicit(\rbar|_{G_{F_w}})$.
  \end{defn}
  \begin{cor}
    \label{cor: explicit global implies diag} Let $\rbar:G_F\to\GL_n(\Flbar)$
    be a continuous representation satisfying the assumptions of Definition~\ref{defn: serre weights explicit global}.
    Then
    $\Wexplicit(\rbar)\subset\Wdiag(\rbar)$.
  \end{cor}
  \begin{proof}
    This follows immediately from Lemma \ref{lem: local explicit implies diag}.
  \end{proof}
  In the case $n=2$, which we explored more thoroughly in
  \cite{BLGGU2}, $\Wexplicit(\rbar)$ is precisely the set of weights
  for which $\rbar$ is modular. We do not conjecture this for $n>2$;
  even for $n=3$ one sees that the set of weights predicted in
  \cite{herzigthesis} is larger than $\Wexplicit(\rbar)$. In fact, we
  expect (see \cite{GHS} for a much more detailed discussion) that the
  set of weights for which $\rbar$ is modular is $\Wcris(\rbar)$, and
  it is easy to see that this set is typically larger than
  $\Wexplicit(\rbar)$. Indeed, by Lemma \ref{lem: equivalence of
    modular of Serre weight and RACSDC lift} and Theorem \ref{thm:
    existence of Galois reps attached to RACSDC}, if $\rbar$ is
  modular of some Serre weight $b$, and $F_b$ is a Jordan-H\"older
  factor of $P_a$ for some Serre weight $a$, then
  $a\in\Wcris(\rbar)$. It is easy to find examples of $a$, $b$ for
  which $b\in\Wexplicit(\rbar)$ but $a\notin\Wexplicit(\rbar)$. On the
  other hand, as explained in \cite{GHS} we believe that
  $\Wcris(\rbar)$ is determined by $\Wexplicit(\rbar)$ and a simple
  combinatorial recipe, so that the weights in $\Wexplicit(\rbar)$ are
  in some sense fundamental.

\subsection{Fontaine-Laffaille theory}\label{sec: Fontaine-Laffaille theory} In applications of our results
it is often useful to have information in the opposite direction;
namely one wishes to have information about $\rbar|_{G_{F_w}}$ at
places $v|p$, given that $\rbar$ is modular of some particular
weight. In the case that $l$ is unramified in $F$ and the weight is
sufficiently far inside the lowest alcove, this can be done by
Fontaine--Laffaille theory. Again, we specialise to the case that $l$
splits completely in $F$.
\begin{lem}
  \label{lem: Fontaine-Laffaille theory} Let $F$ be an imaginary CM field with maximal totally real subfield
  $F^+$, and suppose that $F/F^+$ is unramified at all finite places,
  and that $l$ splits completely in $F$. If $n$ is even, assume that
  $[F^+:\Q]n/2$ is even. Suppose that $l>2$, and that
  $\rbar:G_F\to\GL_n(\Flbar)$ is an irreducible modular representation with
  split ramification. Let  $a\in(\Z^n_+)_0^{\coprod_{w|l}\Hom(k_v,\Flbar)}$ be a
  Serre weight. If $\rbar$ is modular of weight $a$, and $w|l$ is such
  that $a_{w,1}-a_{w,n}\le l-n$, then $a_w\in\Wexplicit_I(\rbar|_{G_{F_w}}^{\semis})$.
\end{lem}
\begin{proof}
  This is a standard application of Fontaine--Laffaille theory. By
  Corollary \ref{cor: modular of some weight implies crystalline lifts
    exist}, $\rbar|_{G_{F_w}}$ has a crystalline lift with Hodge--Tate
  weights $a_{w,1}+n-1,\dots,a_{w,n}$. Since by assumption we have
  $a_{w,1}+n-1-a_{w,n}\le l-1$, the result follows immediately from,
  for example, Proposition 3 of \cite{MR1931205} (note that while this reference
  assumes that the crystalline representation has $\Ql$-coefficients,
  the proof goes through unchanged with $\Qlbar$-coefficients).
\end{proof}

\section{Explicit results for $\GL_3$}\label{GL3 results with l
  split}\subsection{}We now show how one can obtain cleaner results
in the case $n=3$, making use of the fact that the representation
theory of $\GL_3$, while more complicated than that of $\GL_2$, is
rather simpler than that of $\GL_n$ for $n\ge 4$. The following
Lemmas are key to our approach.

\begin{lem}\label{lem: GL_3 decomposition of mod l representations}
  Let $a\in\Z^3_+$ be a Serre weight for $\GL_3(\F_l)$. Then
  \begin{enumerate}
  \item if $l-1\le a_1-a_3$ and $a_1-a_2$, $a_2-a_3\le l-2$, then there is a
    short exact sequence \[0\to F_a\to P_a\to F_b\to 0\]where $b=(a_3+l-2,a_2,a_1-l+2)$.
  \item In all other cases, $P_a=F_a$.
  \end{enumerate}

\end{lem}
\begin{proof}
  This is Proposition 3.18 of \cite{herzigthesis}.
\end{proof}

\begin{lem}
  \label{lem: explicit list of weights for n=3, taking account of niveaus}
  Suppose that $n=3$, and that $a\in\Z^3_+$ is a Serre weight for
  $\GL_3(\Fl)$. If $a\in\Wexplicit(\rbar)$ for some representation $\rbar:\GQl\to\GL_3(\Fl)$, then either $a_1-a_3=l-1$
  and \[\rbar|_{I_\Ql}\cong\omega^{-(a_1+1)}\oplus\omega^{-(a_2+1)}\oplus\omega^{-(a_3+1)},\]or
  there is a permutation $x$, $y$, $z$ of $-(a_1+2)$, $-(a_2+1)$, $-a_3$ such
  that $\rbar|_{I_\Ql}$ is isomorphic to one
  of \[\omega^x\oplus\omega^y\oplus\omega^z,\]\[\omega^x\oplus\omega_2^{y+lz}\oplus\omega_2^{ly+z},\]\[\omega_3^{x+ly+l^2z}\oplus\omega_3^{y+lz+l^2x}\oplus\omega_3^{z+lx+l^2y},\]where
  in the second case we have $(l+1)\nmid ly+z$, and in the third case
  we have $(l^2+l+1)\nmid x+ly+l^2z$.
\end{lem}
\begin{proof}
  This is a simple calculation (it is immediate from the definition
  that $\rbar|_{I_\Ql}$ is of the given form if one ignores the
  divisibility condition, so the only thing to check is when it can be
  the case that $ly+z$ is divisible by $l+1$ or $x+ly+l^2z$ is divisible by $l^2+l+1$).
\end{proof}

\begin{defn}
  Let $a\in\Z^3_+$ be a Serre weight for $\GL_3(\Fl)$. Then we say
  that $a$ is \emph{non-generic} if one of the following three
  conditions hold: $a_1-a_3= l-1$ and $a_1-a_2$,
  $a_2-a_3\le l-2$; or $a_2-a_3=l-2$ and $a_1-a_2\ge 2$; or
  $a_1-a_2=l-2$ and $a_2-a_3 \ge 2$. Otherwise we say that $a$ is
  \emph{generic}.

If $l$ splits completely in $F$ and $a\in(\Z^3_+)_0^{\Hom(F,\Qlbar)}$
is a Serre weight, we say that $a$ is \emph{generic} if for each
$\tau\in\Hom(F,\Qlbar)$ the corresponding Serre weight
$a_\tau\in\Z^3_+$ is generic.
\end{defn}
We remark that this definition of generic is very mild; in
particular, it is much less restrictive than the notion of generic
used in \cite{egh}. (See also Remark~\ref{rem: don't need generic at irred
  places} below.)
\begin{thm}
  \label{thm: explicit result for GL_3 in l split completely case}Let $F$ be an imaginary CM field with maximal totally real subfield
  $F^+$, and suppose that $F/F^+$ is unramified at all finite places,
  and that $l$ splits completely in $F$. Suppose that $l>2$, and that
  $\rbar:G_F\to\GL_3(\Flbar)$ is an irreducible representation with
  split ramification. Assume that

\begin{enumerate}
\item There is a RACSDC automorphic representation $\Pi$ of
  $\GL_3(\A_F)$ of weight $\mu\in(\Z^3_+)_0^{\Hom(F,\C)}$ and level
  prime to $l$ such that
  \begin{itemize}
  \item $\rbar\cong\rbar_{l,\imath}(\Pi)$ (so in particular, $\rbar^c\cong\rbar^\vee\epsilonbar_l^{-2}$).
  \item For each $\tau\in\Hom(F,\C)$, $\mu_{\tau,1}-\mu_{\tau,3}\le l-3$.
     \item $\rbar(G_{F(\zeta_l)})$ is adequate.
  \end{itemize}
\end{enumerate}
 Let $a\in(\Z^3_+)_0^{\coprod_{w|l}\Hom(k_w,\Flbar)}$ be a generic
  Serre weight. Assume that $a\in \Wexplicit(\rbar)$ (so in
  particular, $\rbar|_{G_{F_w}}$ is semisimple for all $w|l$).
Then $\rbar$ is modular of weight $a$.
\end{thm}

\begin{rem}\label{rem: don't need generic at irred places}
  In fact, the proof below shows that it suffices to assume that $a_w$ is generic for
  all places $w|l$ for which $\rbar|_{G_{F_w}}$ has niveau $2$, and that if
  $\rbar|_{G_{F_w}}$ has niveau $1$, then we do not have both $a_1-a_3= l-1$ and $a_1-a_2$,
  $a_2-a_3\le l-2$. In particular, if
  $\rbar|_{G_{F_w}}$ is irreducible for all places $w|l$ (which is the situation
  considered in~\cite{egh}), then we do not need to assume that $a$ is generic.
\end{rem}

\begin{proof}[Proof of Theorem~\ref{thm: explicit result for GL_3 in l split
    completely case}] By Corollaries~\ref{cor: Fontaine-Laffaille potentially
    diagonalizable local lifts implies Serre weight} and~\ref{cor: explicit global implies diag}, $\rbar$ is
  modular of weight $b$ for some Serre weight $b$ with the property
  that $F_b$ is a Jordan-H\"older factor of $P_a$. We wish to show
  that $F_b\cong F_a$. Assume for the sake of contradiction that $F_b\not\cong
  F_a$, so that there is a place $w|l$ with $F_{b_w}\not\cong F_{a_w}$. By Lemma
  \ref{lem: GL_3 decomposition of mod l representations}, we must have $l-1\le a_{w,1}-a_{w,3}$ and
  $a_{w,1}-a_{w,2}$, $a_{w,2}-a_{w,3}\le l-2$, and
  $b_w=(a_{w,3}+l-2,a_{w,2},a_{w,1}-l+2)$.

  Since $l-1\le a_{w,1}-a_{w,3}$, we have
  $b_{w,1}-b_{w,3}=2l-4-(a_{w,1}-a_{w,3})\le l-3$. Thus the assumption
  that $\rbar$ is modular of weight $b$, together with Lemma \ref{lem: Fontaine-Laffaille theory} gives an explicit
  description of the possibilities for $\rbar|_{G_{F_w}}$ (which is
  assumed to be semisimple) in terms of $b_w$, and hence in terms of
  $a_w$. We also have another such description from the assumption
  that $a\in\Wexplicit(\rbar)$. We will now compare these descriptions
  to obtain a contradiction.

It will be useful to note that since we are assuming that
$a_{w,1}-a_{w,2}$, $a_{w,2}-a_{w,3}\le l-2$, and $a_{w,1}-a_{w,3}\ge l-1$ we
have
\numequation\label{eq: inequalities 1}
1\le a_{w,1}-a_{w,2}, a_{w,2}-a_{w,3}\le l-2,
\end{equation}
\numequation\label{eq: inequalities 2}
l-1\le
a_{w,1}-a_{w,3}\le 2l-4,
\end{equation} so that
\numequation\label{eq: congruence 1}
a_{w,1}\not\equiv a_{w,2}\pmod{l-1},
\end{equation}
\numequation\label{eq: congruence 2}
a_{w,2}\not\equiv
a_{w,3}\pmod{l-1},
\end{equation}
\numequation\label{eq: congruence 3}
a_{w,3}\not\equiv a_{w,1}+1\pmod{l-1}.
\end{equation}
\numequation\label{eq: congruence 4}
a_{w,1}-a_{w,3}\not\equiv l-2\pmod{l+1}.
\end{equation}
If $a_{w,1}-a_{w,2}=1$ then the
condition that $a_{w,1}-a_{w,3}\ge l-1$ forces $a_{w,2}-a_{w,3}=l-2$, so that
$a_w$ is not generic. Similarly if $a_{w,2}-a_{w,3}=1$ then $a_w$ is not
generic. Therefore if we assume that $a_{w}$ is generic, we also have \numequation\label{eq: congruence 5}
a_{w,1}\not\equiv a_{w,2}+1\pmod{l-1},
\end{equation}
\numequation\label{eq: congruence 6}
a_{w,2}\not\equiv
a_{w,3}+1\pmod{l-1}.
\end{equation}By the second and third conditions in the definition of genericity, we also
have 
\numequation\label{eq: congruence 7}
a_{w,3}\not\equiv
a_{w,2}+1\pmod{l-1},
\end{equation}
\numequation\label{eq: congruence 8}
a_{w,2}\not\equiv
a_{w,1}+1\pmod{l-1}.
\end{equation}
{\sl Niveau 1 }Suppose firstly that $\rbar|_{G_{F_w}}$ has niveau 1,
i.e.\ that $\rbar|_{I_{F_w}}$ is a direct sum of powers of the mod $l$
cyclotomic character $\omega$. Then since $a\in\Wexplicit(\rbar)$ and
$a$ is generic, we
see from Lemma \ref{lem: explicit list of weights for n=3, taking
  account of niveaus} that \[\rbar|_{I_{F_w}}\cong  \omega^{-(a_{w,1}+2)}\oplus\omega^{-(a_{w,2}+1)}\oplus\omega^{-a_{w,3}}
.\] By Lemma \ref{lem: Fontaine-Laffaille theory}  (applied to $F_b$), we see that we also have  \[\rbar|_{I_{F_w}}\cong  \omega^{-(a_{w,3}+1)}\oplus\omega^{-(a_{w,2}+1)}\oplus\omega^{-(a_{w,1}+1)}
.\]Thus $a_{w,3}\equiv a_{w,1}+1\pmod{l-1}$, contradicting~(\ref{eq: congruence 3}).

{\sl Niveau 2 }Suppose next that $\rbar|_{G_{F_w}}$ has niveau 2,
i.e.\ that $\rbar|_{I_{F_w}}$ is a direct sum of a power of the mod $l$
cyclotomic character $\omega$ and characters $\omega_2^n$, $\omega_2^{ln}$
for some $n$ with $(l+1)\nmid n$, where $\omega_2$ is a choice of fundamental character of
niveau 2. Then since $a\in\Wexplicit(\rbar)$, we
see from Lemma \ref{lem: explicit list of weights for n=3, taking account of niveaus} that $\rbar|_{I_{F_w}}$ is isomorphic to one of the following:\[
  \omega^{-(a_{w,1}+2)}\oplus\omega_2^{-(a_{w,2}+1+la_{w,3})}\oplus\omega_2^{-(l(a_{w,2}+1)+a_{w,3})}
\]\[
  \omega^{-(a_{w,2}+1)}\oplus\omega_2^{-(a_{w,1}+2+la_{w,3})}\oplus\omega_2^{-(l(a_{w,1}+2)+a_{w,3})}
\]\[
  \omega^{-a_{w,3}}\oplus\omega_2^{-(a_{w,1}+2+l(a_{w,2}+1))}\oplus\omega_2^{-(l(a_{w,1}+2)+a_{w,2}+1)}
\] By Lemma \ref{lem: Fontaine-Laffaille theory} (applied to $F_b$), we see that we also have that $\rbar|_{I_{F_w}}$ is isomorphic to one of the following:\[
  \omega^{-(a_{w,1}+1)}\oplus\omega_2^{-(a_{w,2}+1+l(a_{w,3}+l))}\oplus\omega_2^{-(l(a_{w,2}+1)+a_{w,3}+l)}
\]\[
  \omega^{-(a_{w,2}+1)}\oplus\omega_2^{-(a_{w,1}-l+2+l(a_{w,3}+l))}\oplus\omega_2^{-(l(a_{w,1}-l+2)+(a_{w,3}+l))}
\]\[
  \omega^{-(a_{w,3}+1)}\oplus\omega_2^{-(a_{w,1}-l+2+l(a_{w,2}+1))}\oplus\omega_2^{-(l(a_{w,1}-l+2)+a_{w,2}+1)}
\]Comparing the powers of $\omega$ and using~(\ref{eq: congruence 1})--(\ref{eq: congruence 8}), the only possibility is that we 
simultaneously have
 \begin{align*}\rbar|_{G_{F_w}}&\cong
  \omega^{-(a_{w,2}+1)}\oplus\omega_2^{-(a_{w,1}+2+la_{w,3})}\oplus\omega_2^{-(l(a_{w,1}+2)+a_{w,3})},\\
\rbar|_{G_{F_w}}&\cong
  \omega^{-(a_{w,2}+1)}\oplus\omega_2^{-(a_{w,1}-l+2+l(a_{w,3}+l))}\oplus\omega_2^{-(l(a_{w,1}-l+2)+(a_{w,3}+l))}
.\end{align*}
There are now two possibilities to examine. Firstly it
could be the case that \[a_{w,1}+2+la_{w,3}\equiv a_{w,1}-l+2+l(a_{w,3}+l)\pmod{l^2-1};\]but
this implies that $l^2-l\equiv 0\pmod{l^2-1}$, a contradiction. So we
must have  \[a_{w,1}+2+la_{w,3}\equiv l(a_{w,1}-l+2)+(a_{w,3}+l)\pmod{l^2-1}.\]This simplifies
to $a_{w,1}-a_{w,3}\equiv l-2\pmod{l+1}$, contradicting~(\ref{eq: congruence 4}).
  
{\sl Niveau 3 }Suppose finally that $\rbar|_{G_{F_w}}$ has niveau 3,
i.e.\ that $\rbar|_{I_{F_w}}$ is of the form $\omega_3^n\oplus\omega_3^{ln}\oplus\omega_3^{l^2n}$
for some $n$ with $(l^2+l+1)\nmid n$, where $\omega_3$ is a choice of fundamental character of
niveau 3. Then since $a\in\Wexplicit(\rbar)$, we
see that $\rbar|_{I_{F_w}}$ is isomorphic to one of the following:\[
  \omega_3^{-(a_{w,1}+2+l(a_{w,2}+1)+l^2a_{w,3})}\oplus\omega_3^{-(a_{w,2}+1+la_{w,3}+l^2(a_{w,1}+2))}\oplus\omega_3^{-(a_{w,3}+l(a_{w,1}+2)+l^2(a_{w,2}+1))}
\] \[
  \omega_3^{-(a_{w,1}+2+la_{w,3}+l^2(a_{w,2}+1))}\oplus\omega_3^{-(a_{w,3}+l(a_{w,2}+1)+l^2(a_{w,1}+2))}\oplus\omega_3^{-(a_{w,2}+1+l(a_{w,1}+2)+l^2a_{w,3})}
\]On the other hand, by Lemma \ref{lem: Fontaine-Laffaille theory}
(applied to $F_b$) we also have that $\rbar|_{I_{F_w}}$ is isomorphic to one of the following:\[
  \omega_3^{-(a_{w,1}-l+2+l(a_{w,2}+1)+l^2(a_{w,3}+l))}\oplus\omega_3^{-(a_{w,2}+1+l(a_{w,3}+l)+l^2(a_{w,1}-l+2))}\oplus\omega_3^{-(a_{w,3}+l+l(a_{w,1}-l+2)+l^2(a_{w,2}+1))}
\] \[
  \omega_3^{-(a_{w,1}-l+2+l(a_{w,3}+l)+l^2(a_{w,2}+1))}\oplus\omega_3^{-(a_{w,3}+l+l(a_{w,2}+1)+l^2(a_{w,1}-l+2))}\oplus\omega_3^{-(a_{w,2}+1+l(a_{w,1}-l+2)+l^2(a_{w,3}+l))}
\]Examining the exponents in these expressions, we obtain
$12$ possible congruences $\pmod{l^3-1}$, each of which we will now
show yields a contradiction. In each case below we derive a congruence modulo
$l^2+l+1$ or $l^3-1$, and it is easy to see in each case that the
inequalities~(\ref{eq: inequalities 1}) and~(\ref{eq: inequalities 2}) imply
that the congruence has no solutions.

\begin{enumerate}
\item $a_{w,1}+2+l(a_{w,2}+1)+l^2a_{w,3}\equiv a_{w,1}-l+2+l(a_{w,2}+1)+l^2(a_{w,3}+l)\pmod{l^3-1}$. This
  simplifies to $l^2-1\equiv 0\pmod{l^3-1}$, a contradiction.
\item $a_{w,1}+2+l(a_{w,2}+1)+l^2a_{w,3}\equiv a_{w,1}-l+2+l(a_{w,3}+l)+l^2(a_{w,2}+1)\pmod{l^3-1}$. This
  simplifies to $a_{w,2}-a_{w,3}+2\equiv 0\pmod{l^2+l+1}$, a contradiction.
\item $a_{w,1}+2+l(a_{w,2}+1)+l^2a_{w,3}\equiv a_{w,2}+1+l(a_{w,1}-l+2)+l^2(a_{w,3}+l)\pmod{l^3-1}$. This
  simplifies to $a_{w,1}-a_{w,2}\equiv l\pmod{l^2+l+1}$, a contradiction.
\item $a_{w,1}+2+l(a_{w,2}+1)+l^2a_{w,3}\equiv a_{w,2}+1+l(a_{w,3}+l)+l^2(a_{w,1}-l+2)\pmod{l^3-1}$. This
  simplifies to $l(a_{w,1}-a_{w,3}+3)+(a_{w,1}-a_{w,2}+2)\equiv 0\pmod{l^2+l+1}$, which is
  easily seen to be impossible.
\item $a_{w,1}+2+l(a_{w,2}+1)+l^2a_{w,3}\equiv a_{w,3}+l+l(a_{w,1}-l+2)+l^2(a_{w,2}+1)\pmod{l^3-1}$. This
  simplifies to $(a_{w,1}-a_{w,3})+l(a_{w,2}-a_{w,3})+2\equiv 0\pmod{l^2+l+1}$, which is also
  impossible.
\item $a_{w,1}+2+l(a_{w,2}+1)+l^2a_{w,3}\equiv a_{w,3}+l+l(a_{w,2}+1)+l^2(a_{w,1}-l+2)\pmod{l^3-1}$. This
  simplifies to $(l+1)(a_{w,1}-a_{w,3}+2)+1\equiv 0\pmod{l^2+l+1}$, which is
  impossible.
\item $a_{w,1}+2+la_{w,3}+l^2(a_{w,2}+1)\equiv a_{w,1}-l+2+l(a_{w,2}+1)+l^2(a_{w,3}+l)\pmod{l^3-1}$. This
  simplifies to $l(a_{w,2}-a_{w,3}+1)+1\equiv 0\pmod{l^2+l+1}$, a contradiction.
\item $a_{w,1}+2+la_{w,3}+l^2(a_{w,2}+1)\equiv a_{w,1}-l+2+l(a_{w,3}+l)+l^2(a_{w,2}+1)\pmod{l^3-1}$. This
  simplifies to $l^2-l\equiv 0\pmod{l^3-1}$, a contradiction.
\item $a_{w,1}+2+la_{w,3}+l^2(a_{w,2}+1)\equiv a_{w,2}+1+l(a_{w,1}-l+2)+l^2(a_{w,3}+l)\pmod{l^3-1}$. This
  simplifies to $l(a_{w,2}-a_{w,3}+2)\equiv a_{w,1}-a_{w,2}\pmod{l^2+l+1}$, which is easily
  seen to be impossible.
\item $a_{w,1}+2+la_{w,3}+l^2(a_{w,2}+1)\equiv a_{w,2}+1+l(a_{w,3}+l)+l^2(a_{w,1}-l+2)\pmod{l^3-1}$. This
  simplifies to $a_{w,1}-a_{w,2}+2\equiv 0\pmod{l^2+l+1}$, a contradiction.
\item $a_{w,1}+2+la_{w,3}+l^2(a_{w,2}+1)\equiv a_{w,3}+l+l(a_{w,1}-l+2)+l^2(a_{w,2}+1)\pmod{l^3-1}$. This
  simplifies to $a_{w,1}-a_{w,3}\equiv l-2\pmod{l^2+l+1}$, a contradiction.
\item $a_{w,1}+2+la_{w,3}+l^2(a_{w,2}+1)\equiv a_{w,3}+l+l(a_{w,2}+1)+l^2(a_{w,1}-l+2)\pmod{l^3-1}$. This
  simplifies to $l(a_{w,1}-a_{w,2}+1)+a_{w,1}-a_{w,3}+3\equiv 0\pmod{l^2+l+1}$, which is impossible.
\end{enumerate}
As we have obtained a contradiction in every case, we see that
$F_b\cong F_a$, as required.
\end{proof}

\bibliographystyle{amsalpha}
\bibliography{barnetlambgeegeraghty}

\end{document}